\documentclass[10pt]{amsart}

\usepackage{amsfonts}

\usepackage{graphicx}
\usepackage{amsmath}
\usepackage{setspace}
\usepackage{float}

\usepackage[english]{babel}
\usepackage[left=3cm,right=3cm,top=2.5cm,bottom=3cm]{geometry}

\usepackage{amsmath,amssymb,amsthm,bbm,color,graphics,version}
\usepackage{mathrsfs}
\def\eqn#1{\begin{equation}#1\end{equation}}

\newcommand{\bearno}{\begin{eqnarray*}}
\newcommand{\enarno}{\end{eqnarray*}}

 \def\Xi{X^{\infty}}

\def\Ind#1{I(#1)}
\def\th{\theta}

\def\d{\mathrm{d}}

\def\P{{\mathbb P}}   
\def\E{{\mathbb E}}   \def\R{{\mathbb R}}
  
 \def\BB{\mathcal{B}} 
  
\def\FF{\mathcal{F}}  
 \def\TT{\mathcal{T}}  

\def\Px
{\P_x}

  \def\Pgh{\overline{\P}^{G}}
\def\Egh{\overline{\E}^{G}}

\newcommand{\exit}{{\mbox{\, \vspace{3mm}}}\hfill\mbox{$\square$}}

\newtheorem{Thm}{Theorem}

\newtheorem{Lemma}{Lemma}

\newtheorem{Rem}{Remark} 
 \newtheorem{Problem}{Problem}

\title{Multivariate L\'evy-type drift change detection 
and mortality modeling}

\author{Micha{\l}  Krawiec}
\address{Mathematical Insititute, University of Wroc\l aw, pl. Grunwaldzki 2/4, 50-384 Wroc\l aw, Poland}\email{Michal.Krawiec@math.uni.wroc.pl}
\author{Zbigniew Palmowski}
\address{Faculty of Pure and Applied Mathematics, Wroc\l aw University of Science and Technology, Wroc\l aw, Poland} \email{Zbigniew.Palmowski@pwr.edu.pl}

\thanks{
This work is partially supported by National Science Centre, Poland, under grants
No. 2018/29/B/ST1/00756 (2019-2022) and 2016/23/N/HS4/02106 (2017-2020).}

\date{\today}

\begin{document}
\begin{abstract}
In this paper we give a solution to the quickest drift change detection problem for a multivariate L\'evy process consisting of both continuous (Gaussian) and jump components in the Bayesian approach. We do it for a general 0-modified continuous prior distribution of the change point.
Classically, our criterion of optimality is based on a probability of false alarm and an expected delay of the detection, which is then reformulated in terms of a posterior probability of the change point. We find a generator of the posterior probability, which in case of general prior distribution is inhomogeneous in time. The main solving technique uses the optimal stopping theory and is based on solving a certain free-boundary problem. We also construct a Generelized Shiryaev-Roberts statistic, which can be used for applications. The paper is supplemented by two examples, one of which is further used to analyze Polish life tables (after proper calibration) and detect the drift change in the correlated force of mortality of men and women jointly.

\vspace{3mm}


\end{abstract}

\maketitle

\pagestyle{myheadings} \markboth{\sc M.\ Krawiec --- Z.\ Palmowski} {\sc Multidimensional drift change detection}

\vspace{1.8cm}

\tableofcontents

\newpage

\section{Introduction}
Quickest detection problems, often called also \textit{disorder problems}, arise in various fields of applications of mathematics, such as finance, engineering or economics.
All of them address a question how to detect some changes in observed system in an \textit{optimal} way using statistical methods.
One of the main methods was based on the drift change detection using Bayesian approach; see e.g.
Shiryaev \cite{shiryaev1961problem, shiryaev1963optimum}, where Brownian motion with linear drift was considered and the drift has been changing according to an exponential distribution.
The original problem was reformulated in terms of a free-boundary problem and solved using optimal stopping methods. All details of this analysis are also given in surveys
\cite{shiryaev2006disorder, shiryaev2010quickest} (see also references therein).
Apart from Baysian method, the minimax approach have also been used.
This method is based on identifying the optimal detection time based on so-called cumulative sums (CUSUM) strategy; see e.g.
Page \cite{Page}, Beibel \cite{beibel1996note}, Shiryaev \cite{0036-0279-51-4-L15} or Moustakides \cite{moustakides2004optimality} in the Wiener case, or El Karoui et al. \cite{el2015minimax} in the Poisson case.
Many of these quickest detection problems and used methods are gathered in the book of Poor and Hadjiliadis \cite{poor2009quickest}.
In this paper we choose the first approach.

Our {\it first main goal} is to perform the analysis of the quickest drift change detection problem for multivariate processes, taking into account
the dependence between components. We also allow a general 0-modified continuous prior distribution of the change point.

Most of works on the detection problems in Bayesian setting has been devoted to the one-dimensional processes consisting of only continuous (Gaussian) part or only jumps; see e.g. Beibel \cite{beibel1994bayes}, Shiryaev \cite{shiryaev1963optimum} or \cite[Chap. 4]{shiryaev2007optimal} or Poor and Hadjiliadis \cite{poor2009quickest}.
Only some particular cases of jump models without diffusion component have been already analysed, e.g. by Gal’chuk and Rozovskii \cite{gal1971disorder}, Peskir and Shiryaev \cite{peskir2002solving} or Bayraktar et al. \cite{bayraktar2005standard} for the Poisson process, by Gapeev \cite{gapeev2005disorder} for the compound Poisson process with the exponential jumps or by Dayanik and Sezer \cite{dayanik2006compound}
for more general compound Poisson problem. Later, Krawiec et al. \cite{krawiec2017quickest} allowed observed process to have, apart from diffusion ingredient, jumps as well. This is very important in many applications appearing in actuarial science, finance, etc. Still, all of these results concern one-dimensional case only. This paper removes this limitation.

In addition, we assume that a drift change point $\th$ has a general 0-modified continuous prior distribution $G$. In most works it has been assumed that $\th$ can have only (0-modified) exponential distribution. Such assumption makes the free-boundary problem time-homogenous due to lack of memory property, which is not true in the general case.

The {\it methodology} used in this paper is based on
transferring the detection problem to a certain free-boundary problem.
More formally, in this paper
we consider the process $X=(X_t)_{t\geq 0}$ with
\begin{equation}
\label{eq:X_0_inf}X_t:=\left\{\begin{array}{ll} \Xi_t, & t<\th, \\ \Xi_{\th} + X^{0}_{t-\th}, & t \geq \th, \end{array} \right.
\end{equation}
where $\Xi=(\Xi_t)_{t\geq 0}$ and $X^{0}=(X^{0}_t)_{t\geq 0}$ are both independent jump-diffusion processes taking values in $\R^d$.
We assume that $\Xi$ and $X^{0}$ are related with each other via the exponential change of measure described e.g.
in Palmowski and Rolski \cite{palmowski2002technique}. This change of measure can be seen as a form of the drift change between $\Xi$ and $X^{0}$ with additional change in jump distribution. Later we will see the parameter $r$, which corresponds to the rate (direction) of disorder that can be observed after time $\th$.

Let $\th$ has an atom at zero with mass $x>0$.
We choose the classical optimality criterion based on both probability of a false alarm and a mean delay time. That is, in this paper, we are going to find an optimal detection rule $\tau^*\in \mathcal{T}$ for which the following infimum is attained
\begin{equation*}
V^*(x):=\inf_{\tau\in \mathcal{T}}\left\{\overline{\P}^{G}(\tau<\th) + c\overline{\E}^{G}_x[(\tau-\th)^+]\nonumber\right\},
\end{equation*}
where $\mathcal{T}$ is the family of stopping times and $c>0$ is fixed number.
Measure $\overline{\P}^{G}$ will be formally introduced later. Firstly, we transfer above detection problem into
the following optimal stopping problem
\begin{equation*}
V^*(x)=\inf_{\tau\in\mathcal{T}}\overline{\E}^{G}_x\left[1-\pi_{\tau} + c\int_0^{\tau}\pi_sds\right],
\end{equation*}
for the \textit{a posteriori} probability process $\pi=(\pi_t)_{t\geq 0}$ that also will be formally introduced later. The subscript $x$ associated with $\overline{\E}^{G}$ indicates the starting position of process $\pi$ equal to $x$.
In the next step, using the change of measure technique and stochastic calculus, we can identify the infinitesimal generator of the Markov process $\pi$. This part contains results of independent interest on properties of the posterior process $\pi$, that are related to the multidimensionality of the process $X$. In the classical case with exponential distribution $G$, $\pi$ is time-homogenous with generator $\mathcal{A}$. Finally, we formulate the free-boundary value problem, which in the time-homogenous case is as follows
\begin{equation*}
\begin{array}{cc}
\mathcal{A}f(x)=-cx,& 0\leq x < A^*, \\
f(x)=1-x,& A^*\leq x \leq 1,
\end{array}
\end{equation*}
with the boundary conditions
\begin{equation*}
f(A^{*-})=1-A^* \quad {\rm (continuous\; fit)},
\end{equation*}
\begin{equation*}
f'(A^{*-})=-1 \quad {\rm (smooth\; fit)},
\end{equation*}
\begin{equation*}
f'(0^+)=0 \quad {\rm (normal\; entrance)}
\end{equation*}
for some optimal level $A^*$ which allows to identify the threshold optimal alarm rule as
\begin{equation*}
\tau^*=\inf\{t\geq 0:\pi_t\geq A^*\}.
\end{equation*}
We first generalize above free-boundary problem and then solve it for two basic models: two-dimensional Brownian motion and two-dimensional Brownian motion with downward exponential jumps.

Our {\it second main goal} is to apply the solution of above multivariate detection problem to the analysis of correlated {\it change of drift in force of mortality} of men and women.
The life expectancies for men and women are widely recognized as dependent on each other. For example, married people live statistically longer than single ones. Since many insurance products are engineered for marriages or couples it is crucial to detect the change of mortality rate of marriages.
Indeed, the observed improvements of longevity produce challenges related
with the capital requirements that has to be constituted to face this long-term risk and with creating new ways to cross-hedge or to transfer part of the longevity risk to reinsurers or to financial markets.
To do this we need to perform accurate longevity projections
and hence to predict the change of the drift observed in prospective life tables (national or the specific ones used in insurance companies).
In this paper we analyze the Polish life tables for both men and women jointly.
We proceed as follows. We take logarithm of the force of mortality of men and women creating a two-dimensional process, modeled then by a jump-diffusion process. This process consists of observed two-dimensional drift that can be calibrated from the historical data
and a random zero-mean L\'evy-type perturbation.
Based on previous theoretical work we
construct a statistical and numerical procedure based on the generalized version of the Shiryaev-Roberts statistic introduced by Shiryaev \cite{shiryaev1961problem, shiryaev1963optimum} and Roberts \cite{roberts1966comparison}, see also
Polunchenko and Tartakovsky \cite{polunchenko2012state}, Shiryaev \cite{shiryaev2002quickest}, Pollak and Tartakovsky \cite{pollak2009optimality} and Moustakides et al. \cite{moustakides2009numerical}.
Precisely, we start from a continuous statistic derived from the solution of the optimal detection problem in continuous time.
Then we take discrete moments $0<t_1<t_2<\ldots<t_N$, construct an auxiliary statistic and raise the alarm when it exceeds certain threshold $A^*$ identified in the first part of the paper.

The set-up used in examples is, however, simplified compared to the theory presented in the previous sections. Applications focus mainly on multi-dimensionality of presented problem, to see how one can analyse mortality of men and women jointly. The distribution of change time $\theta$ is limited to classical (0-modified) exponential.

The paper is organized as follows. In Section \ref{sec:1} we describe basic setting of the problem, introduce main definitions and notation.
In this section we also formulate main theoretical results of the paper.
Section \ref{sec:2} is devoted to the construction of the \textit{Generalized Shiryaev-Roberts} statistic.
To apply it, we first need to find some density processes related to the processes $X$ prior and post the drift change.
This is done in Section \ref{sec:2} as well.
Particular examples are analyzed in Section \ref{sec:4}.
Next, in Section \ref{sec:mortality}, we give an application of the theoretical results to a real data from life tables.
We finish our paper with some technical proofs given in Section \ref{sec:6}.

\section{Model description and main results}\label{sec:1}
The main observable process is a regime-switching $d$-dimensional process $X=(X_t)_{t\geq 0}$. It changes its behavior at a random moment $\theta$ in
the following way:
\begin{equation}\label{Xrt_maindef}
X_t=\left\{\begin{array}{ll} X^{\infty}_t, & t<\th, \\ X^{\infty}_{\th} + X^{0}_{t-\th}, & t \geq \th, \end{array} \right.
\end{equation}
where $X^{\infty}$ and $X^{0}$ are two different independent L\'evy processes related with each other via exponential change of measure specified later.
The random time $\theta$ is independent
of the pre- and post-change random processes.

We assume the following model when the post-change drift equals $r$. The process that we observe after the change of drift
is a $d$-dimensional processes $X^{0} = (X^{0}_t)_{t\geq 0} = (X^{0}_{t,1},\ldots,X^{0}_{t,d})_{t\geq 0}$  defined as
\begin{equation}\label{defXor}
X_t^{0}:=\sigma W_t^{0}+rt+\sum_{k=1}^{N_t^{0}} J_k^{0}-\mu^{0} m^{0} t,
\end{equation}
where 
\begin{itemize}
\item $W^{0}=(W_t^{0})_{t\geq 0}=(W_{t,1}^{0},\ldots,W_{t,d}^{0})^T$ is a vector of standard independent Brownian motions,
\item $\sigma=(\sigma_{i,j})_{i,j=1,\ldots,d}$ is a matrix of real numbers, responsible for the correlation of the diffusion components of $X^{0}_{t,1},\ldots,X^{0}_{t,d}$,
we assume that $\sigma_{ii}>0$ for all $i=1,\ldots,d$,
\item $r=(r_1,\ldots,r_d)^T$ is a vector of an additional drift,
\item $N^{0}=(N_t^{0})_{t\geq 0}$ is a Poisson process with intensity $\mu^{0}$,
\item $(J_k^{0})_{k\geq 1}$ is a sequence of i.i.d. random vectors responsible for jump sizes; we denote each coordinate of  $J_k^{0}$ by $J_{k,i}^{0}$ for $i=1,\ldots,d$ and its distribution by $F^{0}_i$ with mean $m^{0}_i$; we also denote by $F^{0}$ a joint distribution of vector $J_k^{0}$ and by $m^{0}=(m^{0}_1,\ldots,m^{0}_d)^T$ its mean.
\end{itemize}
We assume that all components of $X_t^{0}$ are stochastically independent, i.e. $W_t^{0}$, $N_t^{0}$ and the sequence $(J_k^{0})_{k=1,2,\ldots}$ are independent.

Similarly, we assume that the process that we observe prior the drift change is a $d$-dimensional process $X^{\infty} = (X^{\infty}_t)_{t\geq 0} = (X^{\infty}_{t,1},\ldots,X^{\infty}_{t,d})_{t\geq 0}$ defined as
\begin{equation}\label{defXi}
X_t^{\infty}:=\sigma W_t^{\infty}+\sum_{k=1}^{N_t^{\infty}} J_k^{\infty}-\mu^{\infty} m^{\infty} t,
\end{equation}
where
\begin{itemize}
\item $W^{\infty}=(W_t^{\infty})_{t\geq 0}$ is a vector of standard independent Brownian motions,
\item matrix $\sigma$ is the same as for the process $X^{0}$,
\item $N^{\infty}=(N_t^{\infty})_{t\geq 0}$ is a Poisson process with intensity $\mu^{\infty}$,
\item $(J_k^{\infty})_{k\geq 1}$ is a sequence of i.i.d. random vectors, where each coordinate $J_{k,i}^{\infty}$ of $J_k^{\infty}$ has distribution $F^{\infty}_i$ with mean $m^{\infty}_i$; we also denote by $F^{\infty}$ a joint distribution of vector $J_k^{\infty}$ and by $m^{\infty}=(m^{\infty}_1,\ldots,m^{\infty}_d)^T$ its mean.
\end{itemize}
We denote $W_t=W^\infty_{t\wedge \theta}+W^0_{(t-\theta)^+}$ which is a Brownian motion as well.

To formally construct the model with a drift change described above, we follow the ideas of Zhitlukhin and Shiryaev \cite{zhitlukhin2013bayesian}. Precisely, we consider a filtered measurable space $(\Omega, \FF, \{\FF_t\}_{t\geq 0})$ with a right-continuous filtration $\{\FF_t\}_{t\geq 0}$, on which we define a \textit{stochastic system with disorder} as follows.
First, on a probability space $(\Omega, \FF, \{\FF_t\}_{t\geq 0})$ we introduce two probability measures $\P^{\infty}$ and $\P^{0}$ with their restrictions to $\mathcal{F}_t$ given by $\P^{\infty}_t:=\P^{\infty}|_{\FF_t}$ and $\P^{0}_t:=\P^{0}|_{\FF_t}$. We assume that for each $t\geq 0$ the restrictions $\P^{\infty}_t$ and $\P^{0}_t$ are equivalent.
The measure $\P^{\infty}$ corresponds to the case when there is no drift change in the system at all and $\P^{0}$ describes the measure under which there is a drift $r$ present from the beginning (i.e. from $t=0$).
In the following we assume that both measures correspond to laws of the processes $X^{\infty}$ and $X^{0}$ described above, respectively.
We also introduce a probability measure $\P$ that dominates $\P^{\infty}$ and $\P^{0}$
and such that the restriction $\P_t:=\P|_{\FF_t}$ is equivalent to $\P^{\infty}_t$ and $\P^{0}_t$ for each $t\geq 0$.
We define the Radon-Nikodym derivatives
\begin{equation}\label{Ldef1}
L_t^{0}:=\frac{\d\P_t^{0}}{\d\P_t}, \quad L_t^{\infty}:=\frac{\d\P_t^{\infty}}{\d\P_t}.
\end{equation}
Furthermore, for $s\in(0,\infty)$ we define
\begin{equation}\label{Ltsr}
L_t^{(s)}:=L_t^{\infty}\Ind{t<s}+\frac{L_{s^-}^{\infty}}{L_{s^-}^{0}}L_t^{0}\Ind{t\geq s}.
\end{equation}
Finally, for any fixed $s\in(0,\infty)$, taking a consistent family of probability measures $(\P_t^{(s)})_{t\geq 0}$ defined via
\begin{equation*}
\frac{\d \P_t^{(s)}}{\d\P_t}=L_t^{(s)}.
\end{equation*}
by the Kolmogorov's existence theorem we can define measures $\P^{(s)}$ such that $\P^{(s)}|_{\FF_t}=\P^{(s)}_t$. Note that for $t<s$ the following equality holds
\begin{equation*}
\P_t^{\infty}=\P_t^{(s)},
\end{equation*}
since disorder after time $t$ does not affect the behavior of the system before time $t$.

We consider Bayesian framework, that is, we assume that the moment of disorder is a random variable $\th$ with a given distribution function denoted by $G(s)$ on $(\R_+,\BB(\R_+))$. We assume that $G(s)$ is continuous for $s>0$ with right derivative $G'(0)>0$.
We define all quantities on an extended filtered probability space $(\overline{\Omega},\overline{\FF},\{\overline{\FF_t}\}_{t\geq 0},\overline{\P}^{G})$ such that
\begin{equation}\label{fullspace}
\overline{\Omega}:=\Omega\times\R_+,\quad \overline{\FF}:=\FF\otimes\BB(\R_+),\quad \overline{\FF}_t:=\FF_t\otimes
\{\emptyset,\R_+\}.
\end{equation}
Measure $\overline{\P}^{G}$ is defined for $A\in\FF$ and $B\in\BB(\R_+)$ as follows
\begin{equation*}
\overline{\P}^{G}(A\times B):=\int_B\P^{(s)}(A)\d G(s).
\end{equation*}
Observe that measure $\overline{\P}^{G}$ describes formally the process $X$ defined in \eqref{Xrt_maindef}.

In the problem of the quickest detection we are looking for an optimal stopping time $\tau^*$ that minimizes certain optimality criterion.
We consider a classical criterion, which incorporates both the probability of false alarm and the mean delay time.
Let $\TT$ denote the class of all stopping times with respect to the filtration $\{\overline{\FF_t}\}_{t\geq 0}$.
Our problem can be stated as follows:
\begin{Problem}\label{Prob:crit1}
For a given $c>0$ calculate the optimal value function
\eqn{\label{eq:crit1}V^*(x)=\inf_{\tau\in\TT}\{\Pgh(\tau<\th) + c\Egh[(\tau-\th)^+]\}}
and find the optimal stopping time $\tau^*$ for which above infimum is attained.
\end{Problem}
Above $\Egh$ means the expectation with respect to $\Pgh$.
The key role in solving this problem plays
\textit{a posterior probability process} $\pi=(\pi_t)_{t\geq 0}$ defined as
\begin{equation}\label{piExtended}
\pi_t:=\overline{\P}^{G}(\th\leq t|\overline{\FF}_t).
\end{equation}
We denote $x:=\pi_0=G(0)$ and add a subscript $x$ to $\Egh_x$ to emphasize it. Using this posterior probability, one can reformulate criterion (\ref{eq:crit1}) into the following, equivalent form:
\begin{Problem}\label{Prob:critExtended}
For a given $c>0$ find the optimal value function
\begin{equation*}\label{eq:opt_fun}
V^*(x)=\inf_{\tau\in\TT}\Egh_x\left[1-\pi_{\tau}+c\int_0^{\tau}\pi_s\d s\right]
\end{equation*}
and the optimal stopping time $\tau^*$ such that
\begin{equation*}\label{eq:opt_time}
V^*(x)=\Egh_x\left[1-\pi_{\tau^*}+c\int_0^{\tau^*}\pi_s\d s\right].
\end{equation*}
\end{Problem}
That is, formally, the following result holds true.
\begin{Lemma}\label{lem:crit}
The criterion given in Problem \ref{Prob:crit1} is equivalent to the criterion given in Problem \ref{Prob:critExtended}.
\end{Lemma}
Although the proof follows classical arguments, we added it in Section \ref{sec:6} for completeness.

Below we formulate the main theorem that connects Problem \ref{Prob:critExtended} to the particular free-boundary problem. It is based on the general optimal stopping theory in the similar way as Theorem 1 in Krawiec et al. \cite{krawiec2017quickest}, which it extends. However, for the general (continuous for $s>0$ with right derivative $G'(0)>0$) distribution $G(s)$ of the moment $\theta$, the optimal stopping problem and its solution are time-dependent. The problem reduces to time-independent case for the (0-modified) exponential distribution $G$. We will prove it in Section \ref{sec:6}.

\begin{Thm}\label{mainresult}
Let $\left(\frac{\partial}{\partial t}+\mathcal{A}\right)$ be a Dynkin generator of the Markov process $(t,\pi_t)_{t\geq 0}$.
Then the optimal value function $V^*(x)$ from the Problem \ref{Prob:critExtended} equals $f_0(x)$, where $f_t(x)$
solves the free-boundary problem
\begin{equation}\label{system}
\begin{array}{cc}
(\frac{\partial}{\partial t}+\mathcal{A}) f_t(x)=-cx,& 0\leq x < A^*(t), \\
f_t(x)=1-x,& A^*(t)\leq x \leq 1,
\end{array}
\end{equation}
with the boundary conditions
\begin{equation}\label{contfit}
f_t(A^{*}(t)-)=1-A^*(t) \quad {\rm (continuous\; fit)},
\end{equation}
\begin{equation}\label{smoothfit}
f^\prime_t(A^{*}(t)-)=-1 \quad {\rm (smooth\; fit)}.
\end{equation}
Furthermore, the optimal stopping time for the Problem \ref{Prob:critExtended} is given by
\begin{equation}\label{ost}
\tau^*=\inf\{t\geq 0:\pi_t\geq A^*(t)\}.
\end{equation}
If $G$ is the (0-modified) exponential distribution, then $V^*(x)$ solves above free-boundary problem for the unique point
$A^*~\in~(0,1]$ not depending on time with the optimal stopping time given by
\begin{equation}\label{ost_exp}
\tau^*=\inf\{t\geq 0:\pi_t\geq A^*\}.
\end{equation}
Further, in this case $f_t(x)=f_0(x)=f(x)$ and additionally the following condition holds
\begin{equation}\label{norment}
f^\prime(0+)=0 \quad {\rm (normal\; entrance)}.
\end{equation}
\end{Thm}

See also Peskir and Shiryaev \cite[Chap. VI. 22]{peskir2006optimal},
Krylov \cite[p. 41]{Krylov}, Strulovici and Szydlowski \cite[Thm. 4]{ss} and \cite{Tiziano} for details.

It is known that the Dynkin generator is an extension of an infinitesimal generator in the sense of their domains.
Following \cite[Chap. III]{peskir2006optimal} and discussion done on page 131 of \cite{peskir2006optimal}
(see also the proof of \cite[Prop. 2.6]{Jacka})
we can conclude that the optimal value function $V^*(t,x)$ satisfies \eqref{system}
where $\left(\frac{\partial}{\partial t} + \mathcal{A}\right)$ is an infinitesimal generator
as long as there exists unique solution
of \eqref{system} lying in the domain of infinitesimal generator.

Now to formulate properly above free-boundary problem, we have to identify the infinitesimal generator $\left(\frac{\partial}{\partial t} + \mathcal{A}\right)$ and its domain. They are given in next theorem.
We use notation $f_t(x) = f(t,x)$ for functions $f: ([0,\infty), [0,1]) \to \mathbb{R}$.

\begin{Thm}\label{generator}
The infinitesimal generator of the Markov process $(t,\pi_t)_{t\geq 0}$ is given by $\frac{\partial}{\partial t}f_t(x) +\mathcal{A} f_t(x)$ for
\begin{multline}\label{genFinal}
\mathcal{A}f_t(x) :=f_t'(x)\bigg( -(1-x)(\log(1-G(t)))' + x(1-x)(\mu^{\infty}-\mu^{0}) \bigg)\\
+ \frac{1}{2}f_t''(x)x^2(1-x)^2\sum_{i=1}^d\sum_{j=1}^d z_{r,i}z_{r,j}(\sigma\sigma^T)_{ij} \\
+ \int_{\R^d}\left[f_t\left(\frac{x\exp\left\{\sum_{i=1}^dz_{r,i}u\right\}}{x\left(\exp\left\{\sum_{i=1}^dz_{r,i}u\right\}-1\right)+1}\right)-f_t(x)\right] \\
\cdot\left[(1-x)\mu^{\infty}\d F^{\infty}(u)+x\mu^{0}\d F^{0}(u)\right]
\end{multline}
and for functions $f_t\in \mathcal{C}^2$. If $G(s)$ is the (0-modified) exponential distribution, then $(\pi_t)_{t\geq 0}$ is a Markov process with generator $\mathcal{A}$ given as above with term $-(1-x)(\log(1-G(t)))'$ substituted by $G'(0)$ for functions $f_t(x)=f(x)$ not depending on $t\geq 0$.
\end{Thm}
We will prove this theorem later in Section \ref{sec:6}.

Assume that we can find unique solution of \eqref{system}-\eqref{smoothfit} in the class $\mathcal{C}^2$ with $\mathcal{A}$
given in \eqref{genFinal} then by above considerations it follows that
this solution equals the value function $V^*(t,x)$.
Therefore in the final step we
focus on the simple time-homogeneous case of exponential time change case, then
we solve uniquely \eqref{system}-\eqref{smoothfit} for some specific choice of model parameters, and finally,
find the optimal threshold $A^*$ and hence the optimal alarm time.
This allows us to construct a Generalized Shiryaev-Roberts statistic in this general set-up.
Later we apply it to detect the changes of drift in joint (correlated) mortality of men and women based on life tables.

\section{Generalized Shiryaev-Roberts statistic}\label{sec:2}
Following Zhitlukhin and Shiryaev \cite{zhitlukhin2013bayesian} and Shiryaev \cite[II.7]{shiryaev1996probability} and using the generalized Bayes theorem, the following equality for process $\pi$ defined in \eqref{piExtended} is satisfied
\begin{equation}\label{pirep2}
\pi_t =  \frac{\int_0^tL_t^{(s)} \d G(s)}{\int_0^{\infty}L_t^{(s)} \d G(s)}.
\end{equation}
We will give another representation of the process $\pi$ in terms of the process $L=(L_t)_{t\geq 0}$ defined by
\begin{equation}\label{defLtr}
L_t:=\frac{L_t^{0}}{L_t^{\infty}}=\frac{\d\P_t^{0}}{\d\P_t^{\infty}}.
\end{equation}
To find above Radon-Nikodym derivative such that process $X$ defined in (\ref{Xrt_maindef}) indeed admits representation (\ref{defXor})
under the measure $\P^{0}$ and (\ref{defXi}) under $\P^{\infty}$, we assume that for given $r=(r_1,\ldots,r_d)\in\R^d$ the following relation holds
\begin{equation}\label{jumprel}
\forall_{x\in\R^d}\quad \mu^{0} F^{0}(\d y)=\frac{h_r(x+y)}{h_r(x)}\mu^{\infty} F^{\infty}(\d y),
\end{equation}
where for $x=(x_1,\ldots,x_d)$ the function $h_r(x): \R^d\to \R$ is given by
\begin{equation}\label{funh}
h_r(x)=\exp\left\{\sum_{j=1}^dz_{r,j}x_j\right\}.
\end{equation}
Above the coefficients $z_{r,1}\ldots z_{r,d}$ solve the following system of equations:
\begin{equation}\label{zsystem}
\left\{\begin{array}{lcl}
r_1-\mu^{0}m^{0}_1+\mu^{\infty}m^{\infty}_1&=&\sum_{j=1}^dz_{r,j}(\sigma\sigma^T)_{1,j}, \\
\vdots&\vdots&\vdots \\
r_d-\mu^{0}m^{0}_d+\mu^{\infty}m^{\infty}_d&=&\sum_{j=1}^dz_{r,j}(\sigma\sigma^T)_{d,j}.
\end{array}\right.
\end{equation}
\begin{Thm}\label{thm:Ltr}
Assume that (\ref{jumprel}) holds for given $r\in\R^d$ and that the Radon-Nikodym derivative $L=(L_t)_{t\geq 0}$ defined in (\ref{defLtr}) is given by
\begin{equation}\label{mainLtr}
L_t=\exp\left\{\sum_{i=1}^dz_{r,i}(X_{t,i}-X_{0,i})-K_rt\right\},
\end{equation}
where
\begin{equation}\label{Kr}
K_r=\frac{1}{2}\sum_{i=1}^d\sum_{j=1}^dz_{r,i}z_{r,j}(\sigma\sigma^T)_{i,j}-\sum_{i=1}^dz_{r,i}\mu^{\infty}m^{\infty}_i+\mu^{0}-\mu^{\infty}.
\end{equation}
Then the process $X$ defined in (\ref{Xrt_maindef}) admits representation (\ref{defXor}) under the measure $\P^{0}$ and (\ref{defXi}) under $\P^{\infty}$.
\end{Thm}
The proof will be given in Section \ref{sec:6}.

Having the density process $L$ defined in \eqref{defLtr} identified in above theorem, we introduce an auxiliary process
\begin{equation}\label{defpsi}
\psi_t:=\int_0^t\frac{L_t}{L_{s^-}}\d G(s).
\end{equation}
Then by (\ref{Ltsr}), (\ref{pirep2}) and (\ref{defLtr}) the following representation of $\pi_t$ holds true
\begin{equation}\label{pirep3}
\pi_t=\frac{\psi_t}{\left[\psi_t+\int_t^{\infty}\frac{L_t^{(s)}}{L_t^{\infty}}\d G(s)\right]}=\frac{\psi_t}{\psi_t+1-G(t)},
\end{equation}
where the last equality follows from the definition of $L_t^{(s)}$ in (\ref{Ltsr}) for $t<s$. By the It\^o's formula applied to (\ref{defpsi}) we obtain that
$\psi^{0}_t$ solves the following SDE
\begin{equation}\label{dpsi}
\d \psi_t=\d G(t)+\frac{\psi_{t^-}}{L_{t}}\d L_t.
\end{equation}

The construction of the classical Shiryaev-Roberts statistic (SR) is in detail described and analyzed e.g.
by Shiryaev \cite{shiryaev2002quickest}, Pollak and Tartakovsky \cite{pollak2009optimality} and Moustakides et al. \cite{moustakides2009numerical}.
In this paper we consider Generalized Shiryaev-Roberts statistic (GSR).
We start the whole construction from
taking the discrete-time data $X_{t_i}\in \mathbb{R}^d$
observed in moments $0=t_0<t_1<\ldots<t_n$, where $n$ is a fixed integer. We assume that $t_i-t_{i-1}=1$ for $i=1,\ldots n$.
Let $x_k:=X_{t_k}-X_{t_{k-1}}$ for $k=1,\ldots,n$. Since $X$ is a $d$-dimensional process, $x_k$ is a
$d$-dimensional vector $x_k=(x_{k,1},\ldots,x_{k,d})$.

Considering a discrete analogue of (\ref{defpsi}) we define the following statistic
\begin{equation*}
\widetilde{\psi}_n:= L_n G(0) + \sum_{j=0}^{n-1}\frac{L_n}{L_j}G'(j) = L_n G(0) + \sum_{j=0}^{n-1}\prod_{k=j+1}^n\exp\left\{\sum_{i=1}^dz_{r,i}x_{k,i}-K_r\right\}G'(j),
\end{equation*}
where from equation (\ref{mainLtr}) we take
\begin{equation*}
L_n:=\exp\left\{\sum_{i=1}^dz_{r,i}\sum_{k=1}^nx_{k,i}-K_rn\right\}=\prod_{k=1}^n\exp\left\{\sum_{i=1}^dz_{r,i}x_{k,i}-K_r\right\}
\end{equation*}
for $n>0$ and $L_0=1$.
Above $G(0)=x$ corresponds to an atom at $0$.

For convenience it can be also calculated recursively as follows:
\begin{equation*}
\widetilde{\psi}_{n+1}=(\widetilde{\psi}_n+G'(n))\cdot\exp\left\{\sum_{i=1}^dz_{r,i}x_{n+1,i}-K_r\right\},\quad \widetilde{\psi}_0=x.
\end{equation*}
Recall from Theorem \ref{mainresult} that the optimal stopping time is given by
\begin{equation*}
\tau^*=\inf\{t\geq 0:\pi_t\geq A^*(t)\}
\end{equation*}
for some optimal level $A^*$.
Therefore from identity (\ref{pirep3}) we can introduce the following
Generalized Shiryaev-Roberts statistic
\begin{equation*}
\widetilde{\pi}_n=\frac{\widetilde{\psi}_n}{\widetilde{\psi}_n+1-G(n)}
\end{equation*}
and raise the alarm of the drift change at the optimal time of the form
\begin{equation*}
\widetilde{\tau}^*:=\inf\{n\geq 0:\widetilde{\pi}_n\geq A^*(n)\}.
\end{equation*}

We emphasize that the GSR statistic is more appropriate in longevity modeling analyzed in this article than the standard one (i.e. SR).
Indeed, the classical statistic is a particular case when $\theta$ has an exponential distribution with parameter $\lambda$ tending to $0$. The latter case corresponds to passing with mean value of the change point $\th$ to $\infty$ and hence  it becomes \textit{conditionally uniform}, see e.g. Shiryaev \cite{shiryaev2002quickest}. Still, in longevity modeling it is more likely that life tables will need to be revised more often and therefore keeping dependence on $\lambda>0$ in our statistic seems to be much more appropriate. For the similar reasons we also prefer to fix average moment of drift change $\theta$ instead of fixing the expected moment of the
revision time $\tau$.

To apply above strategy we will focus on the exponential time of drift change. In this case
we have to identify the optimal alarm level $A^*$ in the first step and hence we have to
solve uniquely the free-boundary value problem (\ref{system}) -- (\ref{smoothfit}). We analyze two particular examples in the next section.

\section{Examples}\label{sec:4}
\subsection{Two-dimensional Brownian motion}\label{Ex1}
Consider the process $X$ without jumps (i.e. with jump intensities $\mu^{\infty}=\mu^{0}=0$). In terms of processes $X^{0}$ and $X^{\infty}$ given in (\ref{defXor}) and ({\ref{defXi}) it means that
\begin{equation}\label{ex1model}
X_t^{0}=\sigma W_t^{0}+rt \quad\text{and}\quad X_t^{\infty}=\sigma W_t^{\infty}.
\end{equation}
Assume that
\begin{equation*}
\sigma=\left(\begin{array}{cc}\sigma_1 & 0 \\ \sigma_2\rho & \sigma_2\sqrt{1-\rho^2}\end{array}\right).
\end{equation*}
Then the first coordinate $X^{0}_{t,1}$ is a Brownian motion with drift and with variance $\sigma^2_1$ and the second coordinate $X^{0}_{t,2}$ is also a Brownian motion with drift and with variance $\sigma^2_2$. The correlation of the Brownian motions on both coordinates is equal to $\rho$. Process $X^{\infty}$ has similar characteristics but without any drift.

Next, assume that, conditioned on $\theta>0$, $\theta$ is exponentially distributed with parameter $\lambda>0$, i.e.
\begin{equation*}
\overline{P}^{G}(\theta\leq t) = G(t) = x + (1-x)(1-e^{-\lambda t}), \quad t\geq 0.
\end{equation*}
Then the generator of process $\pi$ according to (\ref{genFinal}) is equal to
\begin{equation}\label{genEx1}
\mathcal{A}f(x)=f'(x)\lambda(1-x)+\frac{1}{2}f''(x)x^2(1-x)^2\left(z_{r,1}^2\sigma_1^2+z_{r,2}^2\sigma_2^2+2z_{r,1}z_{r,2}\sigma_1\sigma_2\rho\right),
\end{equation}
where $z_{r,1}$ and $z_{r,2}$ solve the following system
\begin{equation*}
\left\{\begin{array}{l}
r_{1}=\sum_{j=1}^2z_{r,j}(\sigma\sigma^T)_{1,j}, \\
r_{2}=\sum_{j=1}^2z_{r,j}(\sigma\sigma^T)_{2,j}.
\end{array}\right.
\end{equation*}

Our goal is to solve the boundary value problem (\ref{system}) -- (\ref{smoothfit}) where generator $\mathcal{A}$ is given by  (\ref{genEx1}).
Note that the system (\ref{system}) takes now the following form
\begin{equation*}
\begin{split}
f'(x)\lambda(1-x)+\frac{1}{2}f''(x)x^2(1-x)^2\cdot B=-cx, \quad 0\leq x<A^*,\\
f(x)=1-x, \quad A^*\leq x \leq 1,
\end{split}
\end{equation*}
where
\begin{equation*}
B:=z_{r,1}^2\sigma_1^2+z_{r,2}^2\sigma_2^2+2z_{r,1}z_{r,2}\sigma_1\sigma_2\rho.
\end{equation*}
Observe that above equations allow us to refer to the classical Shiryaev problem, with our constant $B$ included. Hence, from Shiryaev \cite{shiryaev1961problem, shiryaev2006disorder} it follows that solution
of above equation is given by
\begin{equation*}
V^*(x)=\left\{\begin{array}{ll}1-A^*-\int_x^{A^*}y(s)\d s, & x\in[0,A^*) \\
1-x, & x\in[A^*,1],\end{array}\right.
\end{equation*}
where
\begin{equation*}
y(s)=-\frac{2c}{B}\int_0^se^{-\frac{2\lambda}{B}[Z(s)-Z(u)]}\frac{1}{u(1-u)^2}\d u
\end{equation*}
for
\begin{equation*}
Z(u)=\log\frac{u}{1-u}-\frac{1}{u}.
\end{equation*}
The exact values of function $y(x)$ can be found numerically, while the threshold $A^*$ can be found from the equation $y(A^*)=-1$, which is the boundary condition (\ref{smoothfit}).

\subsection{Two-dimensional Brownian motion with one-sided jumps}
 The second example concerns similar 2-dimensional Brownian motion model as in the previous example, but with additional exponential jumps. Assume that $\mu^{\infty},\mu^{0}>0$ and
\begin{equation}\label{expjump}
F^{\infty}(\d y) = \prod_{j=1}^2F_j^{\infty}(\d y) = \prod_{j=1}^2\frac{1}{w_j}e^{-y_j/w_j}\Ind{y_j\geq 0}\d y.
\end{equation}
In other words, jump sizes on each coordinate $j\in\{1,2\}$ of process $X^{\infty}$ are independent of each other and distributed exponentially with mean $w_j>0$. Additionally, we assume as in the previous example that
\begin{equation*}
\overline{P}^{G}(\theta\leq t) = G(t) = x + (1-x)(1-e^{-\lambda t}), \quad t\geq 0.
\end{equation*}

Jump distribution given by (\ref{expjump}) together with Theorem \ref{thm:Ltr} allows us to formulate the following lemma.

\begin{Lemma}\label{lem:Fr}
Assume that jump distribution $F^{\infty}$ of the process $X^{\infty}$ is given by (\ref{expjump}) and jump intensity is equal to $\mu^{\infty}$. Assume also that there exists a vector $z_{r}=(z_{r,1},\ldots,z_{r,d}$) satisfying the system (\ref{zsystem}) such that $(\forall_{1\leq j \leq 2})(|w_jz_{r,j}|<1)$. Then the following distribution function
$F^{r}$ and intensity $\mu^{0}$ satisfy the condition (\ref{jumprel}):
\begin{equation}\label{murFr}
\begin{split}
&F^{r}(\d y) = \prod_{j=1}^2\frac{1-w_jz_{r,j}}{w_j}e^{-y_j/(\frac{w_j}{1-w_jz_{r,j}})}\Ind{y_j\geq 0}\d y, \\
&\mu^{0} = \mu^{\infty}\prod_{j=1}^2\frac{1}{1-w_jz_{r,j}}.
\end{split}
\end{equation}
\end{Lemma}
\begin{proof}
By the combination of (\ref{jumprel}) and (\ref{funh}) we obtain that
\begin{equation*}
\mu^{0}F^{r}(\d y) = e^{\sum_{j=1}^2z_{r,j}y_j}\mu^{\infty}F^{\infty}(\d y)
= \mu^{\infty}\prod_{j=1}^2\frac{1}{w_j}e^{-y_j/(\frac{w_j}{1-w_jz_{r,j}})}\Ind{y_j\geq 0}\d y,
\end{equation*}
which can be rearranged to
\begin{equation*}
\mu^{\infty}\prod_{j=1}^2\frac{1}{1-w_jz_{r,j}}\cdot \prod_{j=1}^2\frac{1-w_jz_{r,j}}{w_j}e^{-y_j/(\frac{w_j}{1-w_jz_{r,j}})}\Ind{y_j\geq 0}\d y.
\end{equation*}
Now it is sufficient to observe that above formula is equal to the product $\mu^{0} F^{r}$ given by (\ref{murFr}) and that $F^{r}$ is indeed a proper distribution by the assumption that $(\forall_{1\leq j \leq 2})$ $(|w_jz_{r,j}|<1)$.
\end{proof}

\begin{Rem}\rm
Considering the jump distributions $F^{\infty}$ and $F^{r}$ given by (\ref{expjump}) and (\ref{murFr}), the system (\ref{zsystem}) consists of equations
\begin{equation*}
r_{k}+\mu^{\infty}m_k^{\infty}-\mu^{0}m_k^{r}-\sum_{j=1}^2z_{r,j}(\sigma\sigma^T)_{k,j}=0, \quad k=1,\ldots,2,
\end{equation*}
where
\begin{equation*}
\mu^{\infty}m^{\infty}_k = \mu^{\infty}w_k
\end{equation*}
and
\begin{equation*}
\mu^{0}m^{r}_k = \mu^{\infty}\frac{w_k}{1-w_kz_{r,k}}\prod_{j=1}^2\frac{1}{1-w_jz_{r,j}}.
\end{equation*}
\end{Rem}

\begin{Rem}\label{rem:Fr}\rm
Distribution $F^{r}$ given by (\ref{murFr}) has similar characteristics to $F^{\infty}$. More precisely: jumps on both coordinates $X^{0}_{t,1}$ and $X^{0}_{t,2}$ are independent, exponentially distributed with means $\frac{w_1}{1-w_1z_{r,1}}$ and $\frac{w_2}{1-w_2z_{r,2}}$, respectively.
\end{Rem}

The generator $\mathcal{A}$ given by (\ref{genFinal}) for jump distributions specified above can be expressed as
\begin{equation}\label{genEx2}
\begin{split}
\mathcal{A}f(x)=&f'(x)\left( \lambda(1-x)+x(1-x)(\mu^{\infty}-\mu^{0})\right) \\
&+\frac{1}{2}f''(x)x^2(1-x)^2\left[z_{r,1}^2\sigma_1^2+z_{r,2}^2\sigma_2^2+2z_{r,1}z_{r,2}\sigma_1\sigma_2\rho\right]-f(x) \\
&+ \int_{[0,\infty)^2} f\left(\frac{x\exp\{\sum_{i=1}^2z_{r,i}y_i\}}{x(\exp\{\sum_{i=1}^2z_{r,i}y_i\}-1)+1}\right) \\
&\cdot\left[(1-x)\mu^{\infty}\prod_{j=1}^2\frac{1}{w_j}e^{-y_j/w_j}+x\mu^{0}\prod_{j=1}^2\frac{1-w_jz_{r,j}}{w_j}e^{-y_j/\frac{w_j}{1-w_jz_{r,j}}}\right]\d y.
\end{split}
\end{equation}
The integral part of $\mathcal{A}$ can be further simplified. For $\alpha_1,\alpha_2>0$ we define the following integrals
\begin{equation*}
I_+^{0}(x):=\int_{[0,\infty)^2}f\left(\frac{x\exp\{\sum_{i=1}^2z_{r,i}y_i^{0}\}}{x(\exp\{\sum_{i=1}^2z_{r,i}y_i^{0}\}-1)+1}\right)\prod_{j=1}^2\alpha_je^{-\alpha_jy_j}\d y
\end{equation*}
and
\begin{equation*}
I_-^{0}(x):=\int_{(-\infty,0]^2}f\left(\frac{x\exp\{\sum_{i=1}^2z_{r,i}y_i^{0}\}}{x(\exp\{\sum_{i=1}^2z_{r,i}y_i^{0}\}-1)+1}\right)\prod_{j=1}^2\alpha_je^{\alpha_jy_j}\d y.
\end{equation*}
\begin{Lemma}\label{lem:int}
Assume that $\alpha_1, \alpha_2, z_{r,1}, z_{r,2} > 0$ and $\frac{\alpha_1}{z_{r,1}}\neq\frac{\alpha_2}{z_{r,2}}$. Then for $x\in(0,1]$,
\begin{equation*}
\begin{split}
I_+^{0}(x)=f(x)&-\frac{\beta_1}{\beta_2-\beta_1}\left(\frac{1-x}{x}\right)^{-\beta_2}\int_x^1f'(v)\left(\frac{v}{1-v}\right)^{-\beta_2}\d v \\
& +\frac{\beta_2}{\beta_2-\beta_1}\left(\frac{1-x}{x}\right)^{-\beta_1}\int_x^1f'(v)\left(\frac{v}{1-v}\right)^{-\beta_1}\d v
\end{split}
\end{equation*}
for $\beta_1=\frac{\alpha_1}{z_{r,1}}$ and $\beta_2=\frac{\alpha_2}{z_{r,2}}$ and
\begin{equation*}
\begin{split}
I_-^{0}(x)=f(x)&+\frac{\beta_1}{\beta_2-\beta_1}\left(\frac{1-x}{x}\right)^{\beta_2}\int_0^xf'(v)\left(\frac{v}{1-v}\right)^{\beta_2}\d v \\
& - \frac{\beta_2}{\beta_2-\beta_1}\left(\frac{1-x}{x}\right)^{\beta_1}\int_0^xf'(v)\left(\frac{v}{1-v}\right)^{\beta_1}\d v.
\end{split}
\end{equation*}
\end{Lemma}

Using similar arguments like in Krawiec et al. \cite{krawiec2017quickest} that there exists unique solution of the system
equations \eqref{system}-\eqref{smoothfit}, hence whole estimation procedure can be applied.

\begin{Rem}\rm
In Lemma \ref{lem:int} we restrict calculations to the case $\beta_1\neq\beta_2$, but similar transformations of the integral may be made for the case $\beta_1=\beta_2$ as well. The difference will appear in the distribution of random variable $S^{0}$ present in the proof (see Section \ref{sec:6}) being the sum of two exponential random variables.
\end{Rem}

Denote $\gamma_i:=\frac{1}{z_{r,i}w_i}$ for $i\in\{1,2\}$. Then from Lemma \ref{lem:int}
the generator $\mathcal{A}$ given in (\ref{genEx2}) can be rewritten as follows
\begin{multline*}
\mathcal{A}f(x)=f(x)\left[(1-x)\mu^{\infty}+x\mu^{0}-1\right] + f'(x)\left( \lambda(1-x)+x(1-x)(\mu^{\infty}-\mu^{0})\right) \\
+\frac{1}{2}f''(x)x^2(1-x)^2\left[z_{r,1}^2\sigma_1^2+z_{r,2}^2\sigma_2^2+2z_{r,1}z_{r,2}\sigma_1\sigma_2\rho\right] \\
-(1-x)^{-\gamma_2+1}x^{\gamma_2}\int_x^1f'(v)\left[\mu^{\infty}\frac{\gamma_1}{\gamma_2-\gamma_1}\left(\frac{v}{1-v}\right)^{-\gamma_2}+\mu^{0}\frac{\gamma_1-1}{\gamma_2-\gamma_1}\left(\frac{v}{1-v}\right)^{-\gamma_2+1}\right]\d v \\
+(1-x)^{-\gamma_1+1}x^{\gamma_1}\int_x^1f'(v)\left[\mu^{\infty}\frac{\gamma_2}{\gamma_2-\gamma_1}\left(\frac{v}{1-v}\right)^{-\gamma_1}+\mu^{0}\frac{\gamma_2-1}{\gamma_2-\gamma_1}\left(\frac{v}{1-v}\right)^{-\gamma_1+1}\right]\d v.
\end{multline*}
Equation $\mathcal{A}f(x)=-cx$ in the free-boundary value problem  can be further simplified to get rid of the integrals and then solved numerically to find the threshold $A^*$. We believe that this particular case may be finally solved numerically in a similar way as in the numerical analysis described in Krawiec et al. \cite{krawiec2017quickest}, since here we obtain equation of the same order and similar characteristics. However, in this article we focus our applications on the previous example, which is used in practice in the next section.

\begin{Rem}\rm
The results of above example are derived under the assumption of positive exponential jumps. However, the whole analysis can be also conducted for negative exponential jumps, i.e. for the distribution
\begin{equation*}
F^{\infty}(\d y) = \prod_{j=1}^2\frac{1}{w_j}e^{y_j/w_j}\Ind{y_j\leq 0}\d y.
\end{equation*}
Then we can use part of Lemma \ref{lem:int} concerning $I_-^{0}(x)$ to derive the generator $\mathcal{A}$ given by
\begin{multline*}
\mathcal{A}f(x)=f(x)\left[(1-x)\mu^{\infty}+x\mu^{0}-1\right] + f'(x)\left( \lambda(1-x)+x(1-x)(\mu^{\infty}-\mu^{0})\right) \\
+\frac{1}{2}f''(x)x^2(1-x)^2\left[z_{r,1}^2\sigma_1^2+z_{r,2}^2\sigma_2^2+2z_{r,1}z_{r,2}\sigma_1\sigma_2\rho\right] \\
+(1-x)^{\gamma_2+1}x^{-\gamma_2}\int_0^xf'(v)\left[\mu^{\infty}\frac{\gamma_1}{\gamma_2-\gamma_1}\left(\frac{v}{1-v}\right)^{\gamma_2}+\mu^{0}\frac{\gamma_1+1}{\gamma_2-\gamma_1}\left(\frac{v}{1-v}\right)^{\gamma_2+1}\right]\d v \\
-(1-x)^{\gamma_1+1}x^{-\gamma_1}\int_0^xf'(v)\left[\mu^{\infty}\frac{\gamma_2}{\gamma_2-\gamma_1}\left(\frac{v}{1-v}\right)^{\gamma_1}+\mu^{0}\frac{\gamma_2+1}{\gamma_2-\gamma_1}\left(\frac{v}{1-v}\right)^{\gamma_1+1}\right]\d v.
\end{multline*}
\end{Rem}

\section{Application to the force of mortality}\label{sec:mortality}
Now we are going to give an important example of applications, which concerns modeling of the force of mortality process.
We will analyze the joint force of mortality for both men and women. We observe this process over the past decades and check if and when there have been significant changes of drift.

To achieve this goal, we introduce two-dimensional process of the force of mortality $\mu:=(\mu_t)_{t\geq 0}=((\mu_t^1,\mu_t^2))_{t\geq 0}$. We interpret this process as follows:
\begin{itemize}
\item the first coordinate $\mu_t^1$ represents force of mortality of men, while the second one $\mu_t^2$ represents force of mortality of women (of course they are correlated),
\item the time $t$ runs through consecutive years of life tables, e.g. if $t=0$ corresponds to the year 1990, then $t=10$ corresponds to the year 2000,
\item the age of people is fixed for a given process $\mu$, i.e. if $\mu_0$ concerns 50-year old men and women, then $\mu_{10}$ also concerns 50-year old men and women, but in another year.
\end{itemize}

The representation of the force of mortality process is given by
\begin{equation}\label{logmu}
\log\mu_t=\log\bar{\mu}_t+X_t,
\end{equation}
where $\log\bar{\mu}_t:=(\log \mu_t^1, \log \mu_t^2)$ is a deterministic part equal to
\begin{equation*}
\log\bar{\mu}_t=a_0+a_1t.
\end{equation*}
Above $a_0=(a_0^1, a_0^2)$ is a known initial force of mortality vector of men and women
and $a_1=(a_1^1, a_1^2)$ is a vector of a historical drift per one year.
It is worth to mention here that our model is similar to the Lee-Carter model (for fixed age $\omega$, cf. \cite{lee1992modeling} ):
\begin{equation*}
\log\mu_{\omega,t}=a_{\omega}+b_{\omega}k_t+\epsilon_{\omega,t},
\end{equation*}
where $a_{\omega}$ is a chosen number, $k_t$ is certain univariate time series and $\epsilon_{\omega,t}$ is a random error. However, Lee-Carter method focuses on modelling the deterministic part of the force of mortality, while our detection procedure concerns controlling the random perturbation in time, precisely the moment when it substantially changes. This model is univariate as well in contrast to our two-dimensional
mortality process.

In our numerical analysis the stochastic part $X_t$ will be modeled by the two-dimensional Brownian motion analyzed in Example \ref{Ex1}. We apply this model to the life tables downloaded from the Statistics Poland website \cite{gus}.

The first step concerns the model calibration.
We start with some historical values of the force of mortality $\hat{\mu}_0,\ldots,\hat{\mu}_n$, where each $\hat{\mu}_i=(\hat{\mu}_{i,1},\hat{\mu}_{i,2})$ is a two-dimensional vector (one coordinate for women and one for men).
We estimate $a_1$ as a mean value of log-increments of $\hat{\mu}_0,\ldots,\hat{\mu}_n$. Precisely,
\begin{equation*}
\hat{a}_1 := \frac{1}{n}\sum_{i=1}^ny_i,
\end{equation*}
where
\begin{equation*}
y_i := \log\hat{\mu}_i - \log\hat{\mu}_{i-1},\quad i=1,\ldots,n.
\end{equation*}

A little more attention is needed to calibrate the stochastic part $X$, which includes correlation.
Denote
\begin{equation}\label{cal:Xi}
\hat{X}_i = \log\hat{\mu}_i - a_0 - \hat{a}_1i, \quad i=0,\ldots,n
\end{equation}
and the increments
\begin{equation}\label{Xinc}
x_i:= \hat{X}_{i+1}-\hat{X}_i, \quad i=1,\ldots,n.
\end{equation}
We estimate $\sigma_1$ as a standard deviation of the vector $(x_{1,1},x_{2,1},\ldots,x_{n,1})$. Similarly, $\sigma_2$ is calculated as a standard deviation of a vector $(x_{1,2},x_{2,2},\ldots,x_{n,2})$. Finally, we calculate $\rho$ as the sample Pearson correlation coefficient of vectors  $(x_{1,1},x_{2,1},\ldots,x_{n,1})$ and $(x_{1,2},x_{2,2},\ldots,x_{n,2})$.

There are still some model parameters that have to be chosen a priori.
In particular, we have to declare the anticipated incoming drift $r$,
the probability $x=\pi_0=\overline{\P}^{G}(\theta=0)$ that the drift change occurs immediately, the parameter $\lambda>0$ of the exponential distribution of $\th$ and parameter $c$ present in criterion stated in the Problem \ref{eq:opt_time}.

We assume their values at the following level:
\begin{itemize}
\item $\lambda=0.1$. It is the reciprocal of the mean value of $\theta$ distribution conditioned to be strictly positive. Such choice reflects the expectation that the drift will change in $10$ years on average.
\item $x (=\overline{\P}^{G}(\theta=0))=0.1$. This parameter should be rather small (unless we expect the change of drift very quickly).
\item $c=0.1$. It is the weight of the mean delay time inside the optimality criterion stated in the Problem \ref{Prob:crit1}. It reflects how large delay we can accept comparing to the risk of false alarm. We have chosen rather small value and connected it to $\lambda$ by choosing $c=\lambda$.
\item Drift incoming after the moment $\th$ -- we have connected the anticipated value of $r$ to $\sigma$ by $r=(\sigma_1,\sigma_2)$. In practice we suggest to adjust the choice of $r$ to the analysis of sensitivity of e.g. price of an insurance contract.
\end{itemize}
In Table 1 
 we sum up all parameters that were used (both calibrated and arbitrary chosen ones)
in the numerical analysis. The calibration interval was set to years 1990 -- 2000.

\begin{table}[H]
\centering\label{tab_cal}
\begin{tabular}{|c|c|c||c|c|c|c|} \hline
 \multicolumn{3}{|c||}{calibrated} & \multicolumn{4}{c|}{arbitrary chosen} \\ \hline
 $\sigma_1$ & $\sigma_2$ & $\rho$ & c & $\lambda$ & $x$ & $r$ \\ \hline
 $0.03$ & $0.02$ & $0.33$ & $0.1$ & $0.1$ & $0.1$ & $(\sigma_1,\sigma_2)$  \\ \hline
\end{tabular}
 \caption{Parameters used to drift change detection}
\end{table}

\begin{figure}[H]
\includegraphics[width=0.68\textwidth]{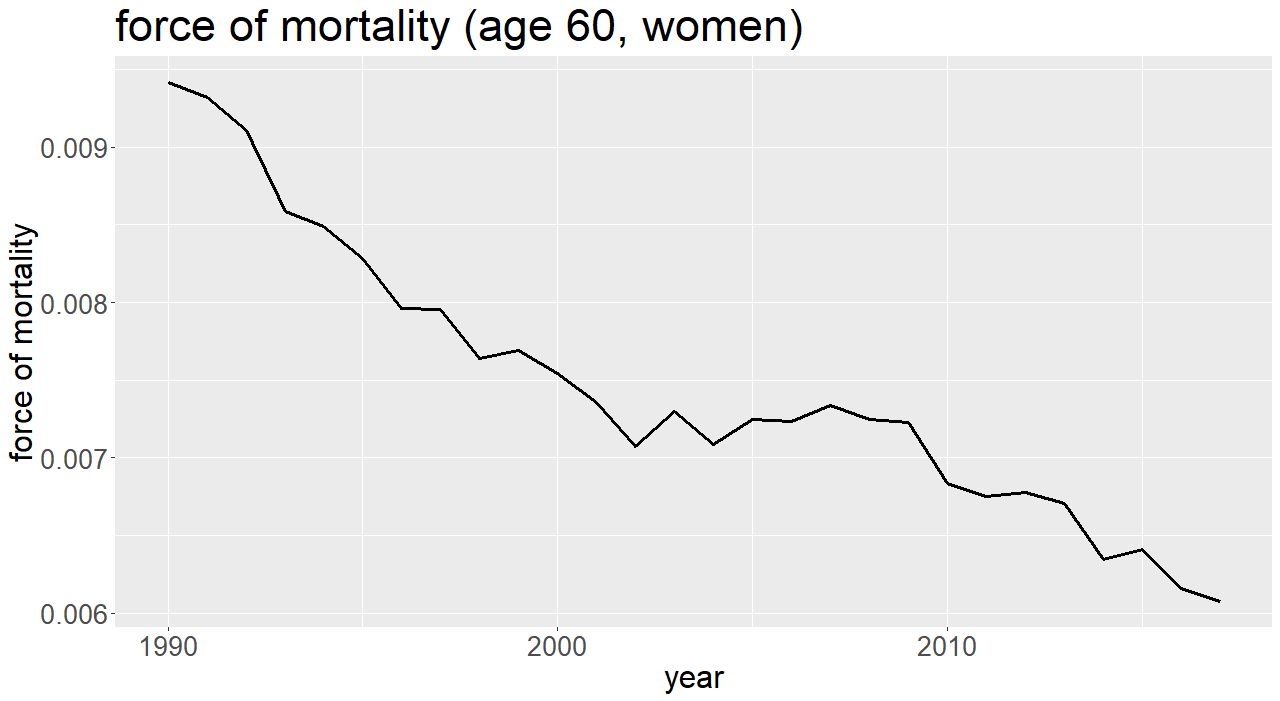}
\centering
\caption{Force of mortality of women aged 60 in 1990-2017}
\label{fig:mu_60_w}
\end{figure}

In the Figure \ref{fig:mu_60_w} we present exemplary plot of the force of mortality for women at age 60 through years 1990 -- 2017. Most of the time it is decreasing, but we can observe a stabilization period around years 2002 -- 2009.
According to (\ref{logmu}) we first take logarithm of the force of mortality, separate deterministic linear part and then model the remaining part by the process $X$ given by (\ref{ex1model}). Figure \ref{fig:brown_60_w} presents historical observations of this remaining part for the same data as in the Figure \ref{fig:mu_60_w}.

\begin{figure}[h]
\includegraphics[width=0.68\textwidth]{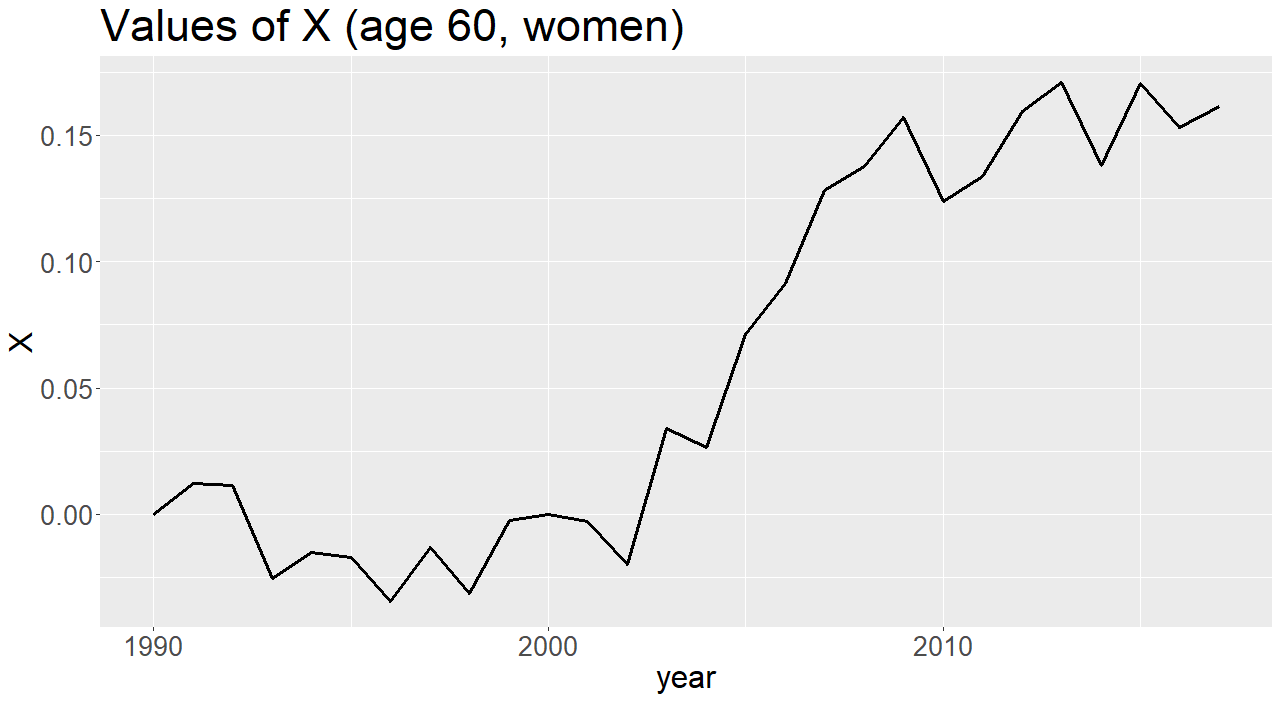}
\centering
\caption{Historical values of $X$ for women aged 60 in 1990-2017}
\label{fig:brown_60_w}
\end{figure}

\begin{figure}[h]
\includegraphics[width=0.48\textwidth]{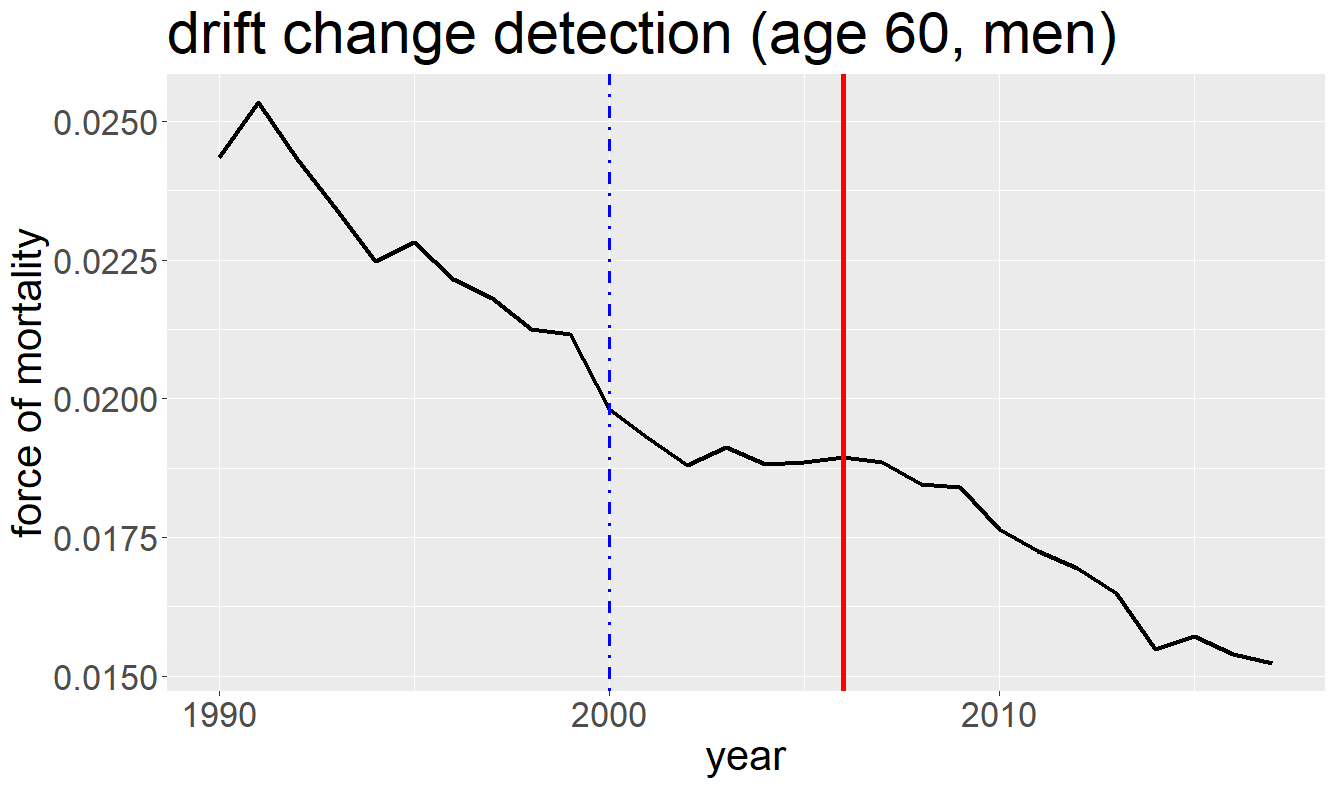}
\includegraphics[width=0.48\textwidth]{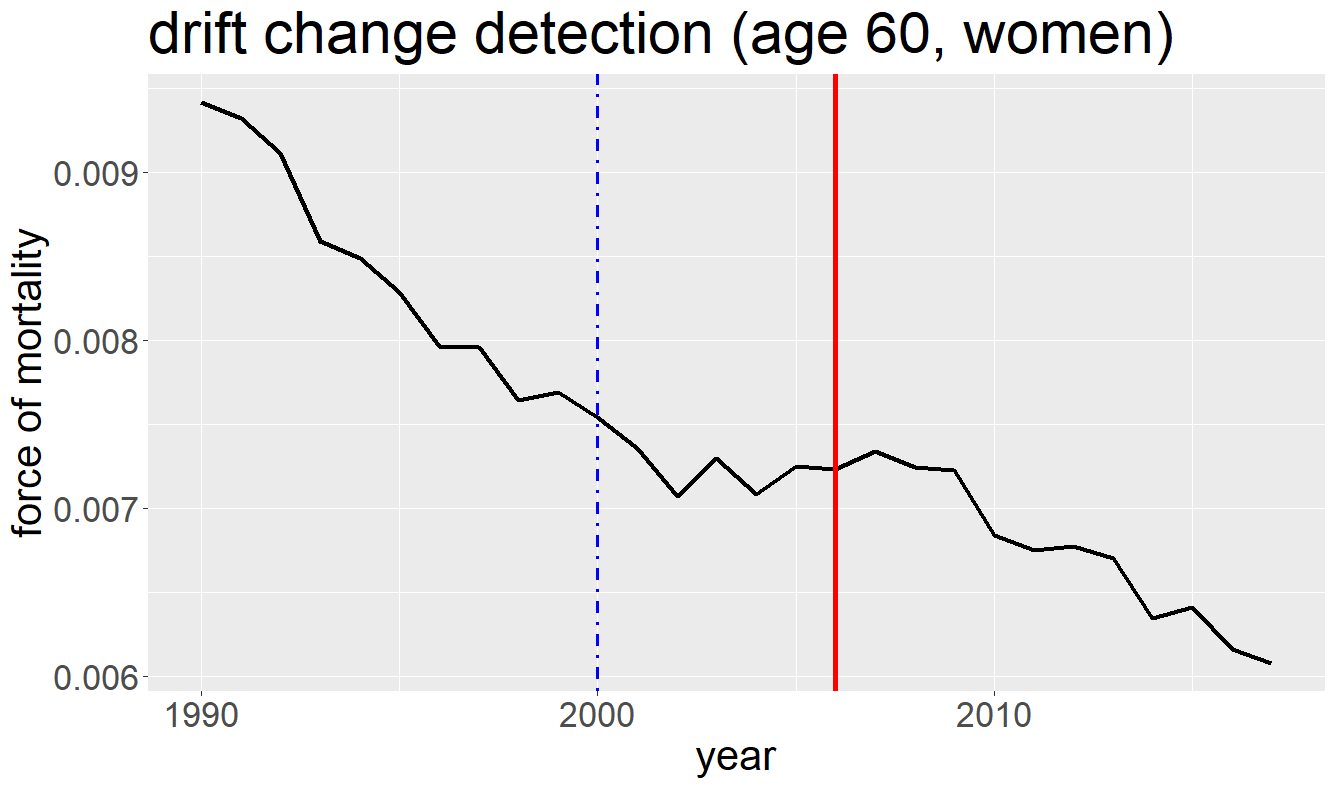}
\includegraphics[width=0.48\textwidth]{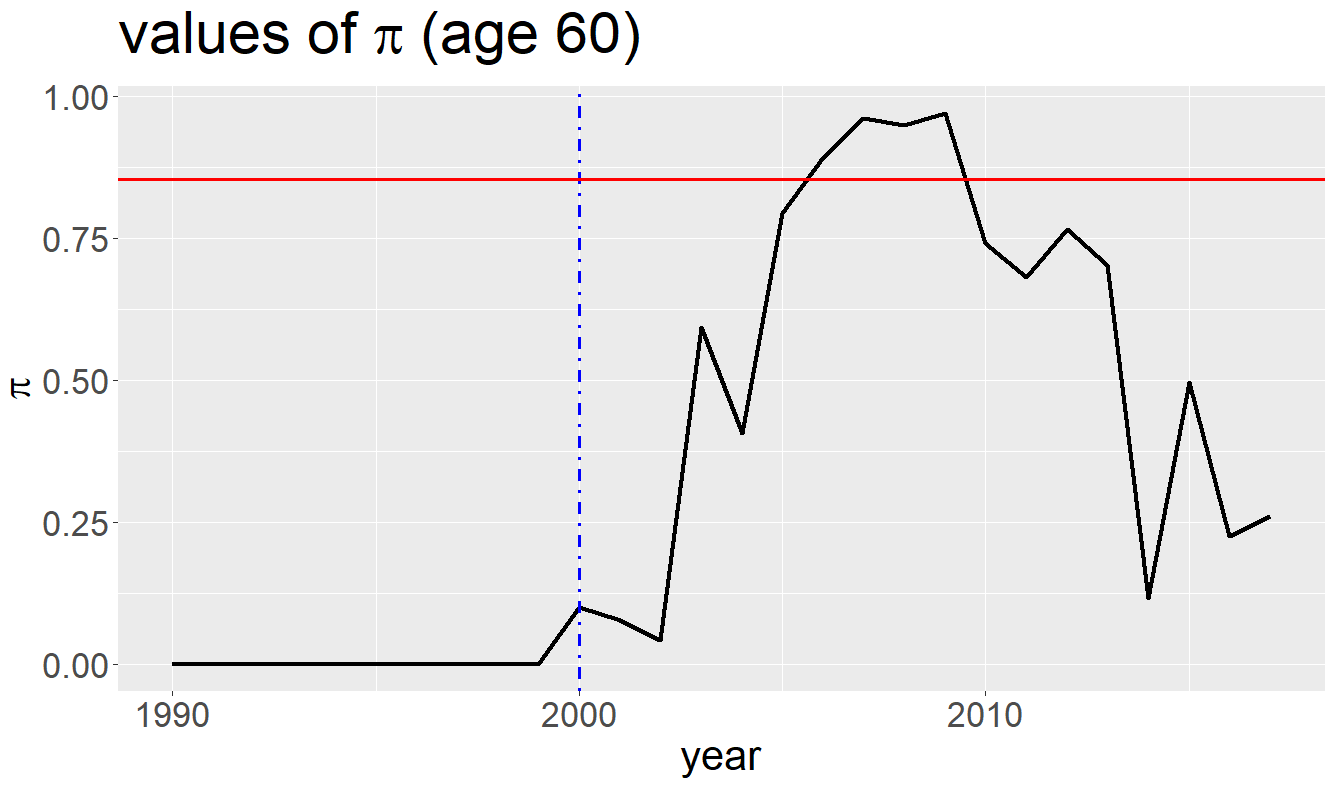}
\centering
\caption{Drift change detection jointly for men and women aged 60 in 1990-2017}
\label{fig:detection_60}
\end{figure}

The results of the detection algorithm for the force of mortality of 60-year old men and women jointly are presented in the Figure \ref{fig:detection_60}. The change of drift for given parameters was detected in year 2006 (red vertical line in the first two plots). The threshold $A^*$ for the optimal stopping time is here equal to $0.85$, which is indicated by the red horizontal line in the third plot presenting values of $\pi=(\pi_t)_{t\in\{1990,\ldots,2017\}}$. 

Note that calibration of parameters (including historical drift) has been done for interval 1990 -- 2000, when the force of mortality was mostly decreasing. After year 2002 it stayed at a stable level for several years, which was detected as a change of drift. This change of behavior is even more evident in the Figure \ref{fig:brown_60_w}, where we can observe that process $X$ is mostly increasing through the years 2002 -- 2009. This example shows that our detection method does not necessarily rise the alarm after the first observed deviation, but rather after it becomes more evident, that the change of drift actually has happened. Therefore, it copes well with cases of gradually changing drift, as long as eventually observed process significantly deviates from the model.

An important note need to be given at the end. This procedure is strongly dependent on parameters chosen to the model -- e.g. post-change drift vector $r$, which was chosen depending on $\sigma_1$ and $\sigma_2$, to give appropriate order of magnitude. A full analysis of the impact of individual parameters on the results would significantly extend this article. However, it may be the subject of further research in further articles, developing the applications of the detection method described here.

\section{Proofs}\label{sec:6}

\paragraph{\bf Proof of Lemma \ref{lem:crit}}
Note that
\eqn{\label{crit:p1}\overline{\P}^{G}(\tau<\th)=\overline{\E}^{G}_x[\overline{\E}^{G}_x[\Ind{\tau<\th}|\overline{\FF}_t]]=\overline{\E}^{G}_x[1-\overline{\P}^{G}(\th\leq\tau|\overline{\FF}_t)]=\overline{\E}^{G}_x[1-\pi_{\tau}].}
Moreover, observe that by  Tonelli's theorem we have:
\begin{equation}\label{crit:p2}\begin{split}
&\overline{\E}^{G}_x[(\tau-\th)^+] = \int_{\R_+}\overline{\E}^{G}_x[(t-\th)^+]\overline{\P}^{G}(\tau\in \d t) = \int_{\R_+}\overline{\E}^{G}_x\left[\int_0^t\Ind{\th\leq s}\d s\right]\overline{\P}^{G}(\tau\in \d t) \\
&= \int_{\R_+}\int_0^t\overline{\E}^{G}_x\left[\overline{\E}^{G}_x\left[\Ind{\th\leq s}|\overline{\FF}_t\right]\right]\d s\overline{\P}^{G}(\tau\in \d t) = \int_{\R_+}\int_0^t\overline{\E}^{G}_x[\pi_s]\d s\overline{\P}^{G}(\tau\in \d t) \\
&= \int_{\R_+}\overline{\E}^{G}_x\left[\int_0^t\pi_s\d s\right]\overline{\P}^{G}(\tau\in \d t) = \overline{\E}^{G}_x\left[\int_0^{\tau}\pi_s\d s\right].
\end{split}\end{equation}
Putting together (\ref{crit:p1}) and (\ref{crit:p2}) completes the proof.
\exit

\paragraph{\bf Proof of Theorem \ref{mainresult}}

We start from the observation that process $((s,\pi_s))_{s\ge 0}$ is Markov which follows from Theorem \ref{generator}.
Let
\begin{equation*}
V^*(t,x):=\inf_{\tau\in\TT, \tau\ge t}\Egh\left\{\left[1-\pi_{\tau}+c\int_0^{\tau}\pi_s\d s\right]\bigg|\pi_t=x\right\}.
\end{equation*}
Then $V^*(x)=V^*(0,x)$.
Moreover, for fixed $t\ge 0$ the optimal value function $V^*(t,x)$ is concave, which follows from \cite[Lem. 3]{krawiec2017quickest}  and the assumption that distribution function $G(t)$ of $\theta$ is continuous for $t>0$.
Observe that from Theorem \ref{generator} it follows that $((s,\pi_s))_{s\ge 0}$ is
stochastically continuous and thus function $(t, x)\rightarrow \Egh\left\{\left[1-\pi_{\tau}+c\int_0^{\tau}\pi_s\d s\right]\bigg|\pi_t=x\right\}$
is continuous for any fixed stopping time $\tau$. Thus from \cite[Rem. 2.10, p. 48]{peskir2006optimal}
we know that the value function $V^*(t,x)$ is lsc.
Let $$C=\{x: V^*(t,x)> 1-x\}$$ be
an open a continuation set
and $D = C^c$ be a stopping set.
From \cite[Cor. 2.9, p. 46]{peskir2006optimal}
we know that $C=[0,A^*(t))$ and that
the stopping rule given by
\begin{equation*}
\tau^*=\inf\{t\geq 0:\pi_t\in D\}
\end{equation*}
is optimal for Problem \ref{Prob:critExtended}.
Moreover, we have \begin{equation}\label{add_assumption}
\Px(\tau^*<\infty)=1
\end{equation}
and
by \cite[Chap. III]{peskir2006optimal}
the optimal value function $V^*(t,x)$ satisfies the following system
\begin{equation}\label{eq:sys_lem}
\left\{\begin{array}{ll}
(\frac{\partial}{\partial t}+\mathcal{A})V^*(t,x)=-cx,&(t,x)\in C,\\ V^*(t,x)=1-x,&(t,x)\in D,\end{array}\right.
\end{equation}
where $\mathcal{A}$ is a Dynkin generator.
Using the same arguments like in the proofs of \cite[Lem. 6 and Lem. 7]{krawiec2017quickest} we
can prove the boundary conditions $f_t(A^{*}(t)-)=1-A^*(t)$ and
$f^\prime_t(A^{*}(t)-)=-1$.

Finally, if $G$ is exponential, then $(\pi_s)_{s\ge 0}$ is Markov by Theorem \ref{generator}.
Now we can do the same arguments but taking simply $V^*(x)$ instead $V^*(t,x)$ above.
Moreover, \eqref{norment} follows from the same arguments like in the proof of \cite[Lem. 6]{krawiec2017quickest}.
This completes the proof.
\exit

\paragraph{\bf Proof of Theorem \ref{generator}}
First, we will find the SDE which is satisfied by process $\pi_t$. By the definition of the process $X$ for each $i=1,\ldots,d$, we get
\begin{equation*}
\d X_{t,i}=\sum_{j=1}^d\sigma_{ij}\d W_{t,j}+\Delta X_{t,i}+r_i\Ind{t\geq\theta}\d t - \left(\mu^{\infty}m_i^{\infty}\Ind{t<\theta}+\mu^{0}m_i^{0}\Ind{t\geq\theta}\right)\d t.
\end{equation*}
Denote the continuous part of the process by an additional upper index $c$. Then
\begin{equation*}
\d\left<X_{t,i}^{c},X_{t,k}^{c}\right> = \sum_{j=1}^d\sigma_{ij}\sigma_{kj}\d t = (\sigma\sigma^T)_{ik}\d t.
\end{equation*}
For the process $L_t$ given in (\ref{mainLtr}), by the It\^o's formula we obtain
\begin{equation*}
\begin{split}
\d L_t &= \left\{\mu^{\infty}-\mu^{0} + \sum_{i=1}^dz_{r,i}(r_i+\mu^{\infty}m_i^{\infty}-\mu^{0}m_i^{0})\Ind{\theta\leq t}\right\}L_t\d t \\
&+ \sum_{i=1}^dz_{r,i}\sum_{j=1}^d\sigma_{ij}L_t^{0} \d W_{t,j} + \Delta L_t,
\end{split}
\end{equation*}
where
\begin{equation*}
\Delta L_t=L_{t^-}\left(\frac{L_t}{L_{t^-}}-1\right) = L_{t^-}\left(e^{\sum_{i=1}^dz_{r,i}\Delta X_{t,i}}-1\right).
\end{equation*}
By (\ref{dpsi}) we conclude that
\begin{equation*}
\begin{split}
\d \psi_t &= \d G(t)+\left\{\mu^{\infty}-\mu^{0}+\sum_{i=1}^dz_{r,i}(r_i+\mu^{\infty}m_i^{\infty}-\mu^{0}m_i^{0})\Ind{\theta\leq t}\right\}\psi_t\d t \\
&+\sum_{i=1}^dz_{r,i}\sum_{j=1}^d\sigma_{ij}\psi_t\d W_{t,j}+\psi_{t^-}\left(e^{\sum_{i=1}^dz_{r,i}\Delta X_{t,i}}-1\right).
\end{split}
\end{equation*}
Recall that by (\ref{pirep3}) we have
\begin{equation*}
\pi_t=\frac{\psi_t}{\psi_t+1-G(t)}.
\end{equation*}
Then, using It\^o's formula once again we obtain
\begin{equation*}
\d \pi_t = \frac{\pi_t(1-\pi_t)}{1-G(t)}\d G(t)+\frac{(1-\pi_t)^2}{1-G(t)}\d\psi_t^{c} - \frac{(1-\pi_t)^3}{(1-G(t)^2}\d\left<\psi^{c},\psi^{c}\right>_t + \Delta\pi_t.
\end{equation*}
Moreover,
\begin{equation*}
\d\left<\psi^{c},\psi^{c}\right>_t=\sum_{i=1}^d\sum_{j=1}^dz_{r,i}z_{r,j}(\sigma\sigma^T)_{ij}(\psi_t)^2\d t.
\end{equation*}
Together with the system of equations (\ref{zsystem}) it produces
\begin{equation*}
\begin{split}
\d\pi_t &= \frac{1-\pi_t}{1-G(t)}\d G(t) + \frac{(1-\pi_t)^2}{1-G(t)}(\mu^{\infty}-\mu)\psi_t\d t \\
&+ \frac{(1-\pi_t)^2}{1-G(t)}\sum_{i=1}^dz_{r,i}\sum_{j=1}^d\sigma_{ij}\psi_t\d W_{t,j} \\
& + \frac{(1-\pi_t)^2}{1-G(t)}\sum_{i=1}^d\sum_{j=1}^dz_{r,i}z_{r,j}(\sigma\sigma^T)_{ij}\Ind{\theta\leq t}\psi_t\d t \\
&- \frac{(1-\pi_t)^3}{(1-G(t))^2}\sum_{i=1}^d\sum_{j=1}^dz_{r,i}z_{r,j}(\sigma\sigma^T)_{ij}(\psi_t)^2\d t + \Delta\pi_t.
\end{split}
\end{equation*}
Jump part of $\pi_t$ equals
\begin{equation*}
\Delta\pi_t = \pi_{t-}\left(\frac{\psi_t}{\psi_{t-}}\frac{\psi_{t-}+1-G(t)}{\psi_t+1-G(t)}-1\right) = \frac{\pi_{t-}\left(\exp\{\sum_{i=1}^dz_{r,i}\Delta X_{t,i}\}-1\right)(1-G(t))}{\psi_{t-}\exp\{\sum_{i=1}^dz_{r,i}\Delta X_{t,i}\}+1-G(t)}.
\end{equation*}
Using the It\^o's formula one more time completes the proof.
\exit

\paragraph{\bf Proof of Theorem \ref{thm:Ltr}}
The proof is based on the technique of exponential change of measure described in Palmowski and Rolski \cite{palmowski2002technique}.

Firstly, we will prove that the process $(L_t)_{t\geq 0}$ satisfies the following representation
\begin{equation}\label{Ltrgeneral}
L_t=\frac{h(X_t)}{h(X_0)}\exp\left(-\int_0^t\frac{(\mathcal{A}^{\infty}h)(X_s)}{h(X_s)}\d s\right)
\end{equation}
for the function $h(x):=h_r(x)$ given in (\ref{funh}), where $\mathcal{A}^{\infty}$ is an extended generator of the process $X$ under $\P^{\infty}$
and $h$ is in its domain since it is twice continuously differentiable.
Then from Theorem 4.2 by Palmowski and Rolski \cite{palmowski2002technique} it follows that the generator of $X$ under $\P^{0}$ is related with $\mathcal{A}^{\infty}$ by
\begin{equation}\label{geneq}
\mathcal{A}^{0} f=\frac{1}{h}\left[\mathcal{A}^{\infty}(fh)-f\mathcal{A}^{\infty}h\right].
\end{equation}

On the other hand, from the definition of the infinitesimal generator or using the Theorem 31.5 in Sato \cite{Sato} it follows that for twice continuously differentiable function $f(x_1,\ldots,x_d):\R^d\to\R$ generators $\mathcal{A}^{\infty}$ and $\mathcal{A}^{0}$ are given by
\begin{equation}\label{geninf}
\begin{split}
\mathcal{A}^{\infty}f(x)&=\frac{1}{2}\sum_{i=1}^d\sum_{j=1}^d\frac{\partial^2f}{\partial x_i\partial x_j}(x)(\sigma\sigma^T)_{i,j}-\sum_{i=1}^d\frac{\partial f}{\partial x_i}(x)\mu^{\infty}m^{\infty}_i  \\
&+ \int_{\R^d}\left(f(x+y)-f(x)\right)\mu^{\infty} F^{\infty}(\d y),
\end{split}
\end{equation}
\begin{equation}\label{genr}
\begin{split}
\mathcal{A}^{0}f(x)&=\frac{1}{2}\sum_{i=1}^d\sum_{j=1}^d\frac{\partial^2f}{\partial x_i\partial x_j}(x)(\sigma\sigma^T)_{i,j}-\sum_{i=1}^d\frac{\partial f}{\partial x_i}(x)\left(\mu^{0}m^{0}_i-r_i\right) \\
&+ \int_{\R^d}\left(f(x+y)-f(x)\right)\mu^{0} F^{0}(\d y).
\end{split}
\end{equation}
For $h_r(x)$ given by (\ref{funh}) we obtain
\begin{equation*}
\frac{h_r(X_t)}{h_r(X_0)}=\exp\left\{\sum_{j=1}^dz_{r,j}\left(X_{t,j}-X_{0,j}\right)\right\}.
\end{equation*}
Further, since
\begin{equation*}
\frac{\partial h_r}{\partial x_i}=z_{r,i}h_r
\end{equation*}
and
\begin{multline*}
\int_{\R^d}\frac{h(X_s+y)-h(X_s)}{h(X_s)}\mu^{\infty}F^{\infty}(dy)=\int_{\R^d}\mu^{0}F^{0}(dy)-\int_{\R^d}\mu^{\infty}F^{\infty}(dy)
= \mu^{0}-\mu^{\infty},
\end{multline*}
then
\begin{multline*}
\frac{(\mathcal{A}^{\infty}h_r)(X_s)}{h(X_s)}=\frac{1}{2}\sum_{i=1}^d\sum_{j=1}^d z_{r,i}z_{r,j}(\sigma\sigma^T)_{i,j}-\sum_{i=1}^dz_{r,i}\mu^{\infty}m^{\infty}_i + \mu^{0}-\mu^{\infty}=K_r.
\end{multline*}
Hence, we obtain
\begin{equation*}
\begin{split}
L_t^{0}&=\exp\left\{\sum_{j=1}^dz_{r,j}(X_{t,j}-X_{0,j})\right\}\cdot \exp\left\{-\int_0^tK_r\d s\right\}\\
&= \exp\left\{\sum_{j=1}^dz_{r,j}(X_{t,j}-X_{0,j})-K_rt\right\}
\end{split}
\end{equation*}
and thus $L_t^{0}$ given in (\ref{mainLtr}) indeed satisfies the representation (\ref{Ltrgeneral}) for function $h_r(x)$ given by (\ref{funh}).

To finish the proof it is sufficient to show that the generator $\mathcal{A}^{0}$ given by (\ref{genr}) indeed coincides with the generator given in (\ref{geneq}) for $h(x)=h_r(x)$. First, by (\ref{Kr}) we get
\begin{equation*}
\frac{1}{h}\left[\mathcal{A}^{\infty}(fh)-f\mathcal{A}^{\infty}h\right]=\frac{\mathcal{A}^{\infty}(fh)}{h}-\frac{f\mathcal{A}^{\infty}h}{h}=\frac{\mathcal{A}^{\infty}(fh)}{h}-fK_r.
\end{equation*}
Second, (\ref{geninf}) produces
\begin{multline*}
\frac{\mathcal{A}^{\infty}(fh)}{h}=\frac{1}{2}\sum_{i=1}^d\sum_{j=1}^d\left(\frac{\partial^2f}{\partial x_i\partial x_j}+\frac{\partial f}{\partial x_i}z_{r,j}+\frac{\partial f}{\partial x_j}z_{r,i}+fz_{r,j}z_{r,i}\right)(\sigma\sigma^T)_{i,j} \\
- \sum_{i=1}^d\left(\frac{\partial f}{\partial x_i}+fz_{r,i}\right)\mu^{\infty}m^{\infty}_i + \int_{\R^d}\frac{f(x+y)h(x+y)-f(x)h(x)}{h(x)}\mu^{\infty}F^{\infty}(dy).
\end{multline*}
Hence
\begin{multline*}
\frac{\mathcal{A}^{\infty}(fh)}{h}-fK_r=\frac{1}{2}\sum_{i=1}^d\sum_{j=1}^d\frac{\partial^2f}{\partial x_i\partial x_j}(\sigma\sigma^T)_{i,j} \\
+\sum_{i=1}^d\frac{\partial f}{\partial x_i}\sum_{j=1}^d\left(z_{r,j}(\sigma\sigma^T)_{i,j}-\mu^{\infty}m^{\infty}_i\right) + \int_{\R^d}(f(x+y)-f(x))\mu^{0}F^{0}(dy).
\end{multline*}
Finally, using the system of equations (\ref{zsystem}) completes the proof.
\exit

\paragraph{\bf Proof of Lemma \ref{lem:int}}
First observe that $I_+^{0}(x)$ is equal to the expectation
\begin{equation}\label{eq:expI+}
\mathbb{E}\left[f\left(\frac{x\exp\{\sum_{i=1}^2z_{r,i}T_i\}}{x(\exp\{\sum_{i=1}^2z_{r,i}T_i\}-1)+1}\right)\right],
\end{equation}
where $T_1$ and $T_2$ are two independent random variables with exponential distributions ${\rm Exp}(\alpha_1)$ and ${\rm Exp}(\alpha_2)$, respectively. Then $z_{r,1}T_1\sim {\rm Exp}\left(\frac{\alpha_1}{z_{r,1}}\right)$, $z_{r,2}T_2\sim {\rm Exp}\left(\frac{\alpha_2}{z_{r,2}}\right)$ and the density of
$S^{0}:=\sum_{i=1}^2z_{r,i}T_i$ is given by
\begin{equation*}
f_{S^{0}}(y)=\frac{\beta_1\beta_2}{\beta_1-\beta_2}\left(e^{-\beta_2y}-e^{-\beta_1y}\right)\Ind{y\geq 0}\d y.
\end{equation*}
Hence, the expectation (\ref{eq:expI+}) equals
\begin{equation*}
\mathbb{E}\left[f\left(\frac{xe^{S^{0}}}{x(e^{S^{0}}-1)+1}\right)\right] = \int_0^{\infty}f\left(\frac{xe^y}{x(e^y-1)+1}\right)\frac{\beta_1\beta_2}{\beta_1-\beta_2}\left(e^{-\beta_2y}-e^{-\beta_1y}\right)\d y.
\end{equation*}
Next, we can integrate above integral by parts to obtain
\begin{eqnarray*}
\lefteqn{f(x)-\int_0^{\infty}f'\left(\frac{xe^y}{x(e^y-1)+1}\right)\frac{\d}{\d y}\left(\frac{xe^y}{x(e^y-1)+1}\right)}
\\&&\qquad \cdot \left(\frac{\beta_1}{\beta_2-\beta_1}e^{-\beta_2y}-\frac{\beta_2}{\beta_2-\beta_1}e^{-\beta_1y}\right)\d y
\end{eqnarray*}
and by substitution $v:=\frac{xe^y}{x(e^y-1)+1}$ (hence $y=\ln\left(\frac{v(1-x)}{x(1-v)}\right)$) we derive
\begin{equation*}
\begin{split}
f(x)&-\frac{\beta_1}{\beta_2-\beta_1}\left(\frac{1-x}{x}\right)^{-\beta_2}\int_x^1f'(v)\left(\frac{v}{1-v}\right)^{-\beta_2}\d v \\ &+\frac{\beta_2}{\beta_2-\beta_1}\left(\frac{1-x}{x}\right)^{-\beta_1}\int_x^1f'(v)\left(\frac{v}{1-v}\right)^{-\beta_1}\d v,
\end{split}
\end{equation*}
which completes the first part of the proof.

The formula for $I_-^{0}(x)$ can be derived by substitution $u:=-y$ to get
\begin{equation*}
I_-^{0}(x)=\int_{(0,\infty)^2}f\left(\frac{x\exp\{-\sum_{i=1}^2z_{r,i}u_i^{0}\}}{x(\exp\{-\sum_{i=1}^2z_{r,i}u_i^{0}\}-1)+1}\right)\prod_{j=1}^2\alpha_je^{-\alpha_ju_i^{0}}\d y,
\end{equation*}
which, by the same arguments as for $I_+^{0}(x)$, is equal to
\begin{equation*}
\mathbb{E}\left[f\left(\frac{xe^{-S^{0}}}{x(e^{-S^{0}}-1)+1}\right)\right] = \int_0^{\infty}f\left(\frac{xe^{-y}}{x(e^{-y}-1)+1}\right)\frac{\beta_1\beta_2}{\beta_1-\beta_2}\left(e^{-\beta_2y}-e^{-\beta_1y}\right)\d y.
\end{equation*}
Integration by parts together with substitution of $v:=\frac{xe^{-y}}{x(e^{-y}-1)+1}$ (hence $y=-\ln\left(\frac{v(1-x)}{x(1-v)}\right)$) gives
\begin{equation*}
\begin{split}
f(x)&+\frac{\beta_1}{\beta_2-\beta_1}\left(\frac{1-x}{x}\right)^{\beta_2}\int_0^xf'(v)\left(\frac{v}{1-v}\right)^{\beta_2}\d v \\
& - \frac{\beta_2}{\beta_2-\beta_1}\left(\frac{1-x}{x}\right)^{\beta_1}\int_0^xf'(v)\left(\frac{v}{1-v}\right)^{\beta_1}\d v
\end{split}
\end{equation*}
which completes the second part of the proof.
\exit

\paragraph{\bf DATA AVAILABILITY STATEMENT}
The datasets analysed during the current study are available in the
repositories {\it http://www.mortality.org/} and {\it stat.gov.pl/en/topics/population/life-expectancy/life-expectancy-in-poland,1,3.html}.


\begin{thebibliography}{10}
\expandafter\ifx\csname url\endcsname\relax
  \def\url#1{\texttt{#1}}\fi
\expandafter\ifx\csname urlprefix\endcsname\relax\def\urlprefix{URL }\fi
\expandafter\ifx\csname href\endcsname\relax
  \def\href#1#2{#2} \def\path#1{#1}\fi



\bibitem{Tiziano}
T. De Angelis and G. Peskir,
Global $C^1$ Regularity of the Value Function
in Optimal Stopping Problems,
Ann. Appl. Probab. (2020).

\bibitem{beibel1994bayes}
M.~Beibel, Bayes problems in change-point models for the {Wiener} process,
  Lecture Notes-Monograph Series, (1994), 1--6.


\bibitem{beibel1996note}
M.~Beibel, et~al., A note on {Ritov's} {Bayes} approach to the minimax property
  of the cusum procedure, The Annals of Statistics 24~(4) (1996) 1804--1812.

\bibitem{bayraktar2005standard}
E.~Bayraktar, S.~Dayanik, I.~Karatzas, The standard {Poisson} disorder problem
  revisited, Stochastic Processes and their Applications 115~(9) (2005)
  1437--1450.

\bibitem{dayanik2006compound}
S.~Dayanik, S.~O. Sezer, Compound {Poisson} disorder problem, Mathematics of
  Operations Research 31~(4) (2006) 649--672.

\bibitem{new1}
S.~Dayanik, C.~Goulding, H.V.~Poor,
Bayesian Sequential Change Diagnosis,
Mathematics of Operations Research 33(2) (2008).

\bibitem{gal1971disorder}
L.~I. Gal’chuk, B.~Rozovskii, The “disorder” problem for a {Poisson}
  process, Theory of Probability \& Its Applications 16(4) (1971) 712--716.

\bibitem{gapeev2005disorder}
P.~V. Gapeev, The disorder problem for compound {Poisson} processes with
  exponential jumps, Annals of Applied Probability 15(1A) (2005) 487--499.

\bibitem{Jacka}
S.~D. Jacka,
Optimal stoppong and the American put,
Mathematical Finance 1(2) (1991), 1--14.


\bibitem{el2015minimax}
N.~El~Karoui, S.~Loisel, Y.~Salhi,
Minimax optimality in robust detection of a
  disorder time in {Poisson} rate,
Annals of Applied Probability 27(4) (2017) 2515--2538.

\bibitem{krawiec2017quickest}
M.~Krawiec, Z.~Palmowski, and L.~Plociniczak,
Quickest drift change detection in {L\'evy}-type force of mortality
  model,
Applied Mathematics and Computation 338 (2018) 432--450.

\bibitem{Krylov}
N. Krylov,
Controlled Diffusion Processes,
Springer Verlag, (1980).


\bibitem{lee1992modeling}
R.~D. Lee and L.~R. Carter,
Modeling and forecasting {US} mortality,
Journal of the American Statistical Association 87(419) (1992) 659--671.

\bibitem{moustakides2004optimality}
G.~V. Moustakides, Optimality of the {CUSUM} procedure in continuous time,
  Annals of Statistics (2004) 302--315.

\bibitem{moustakides2009numerical}
G.~V. Moustakides, A.~S. Polunchenko, A.~G. Tartakovsky, Numerical comparison
  of {CUSUM} and {Shiryaev}--{Roberts} procedures for detecting changes in
  distributions, Communications in Statistics -— Theory and Methods
  38(16-17) (2009) 3225--3239.

\bibitem{Page}
E.~S. Page,
Contnuous inspection schemes,
Biometrika 41(1-2) 1954 100--115.


\bibitem{palmowski2002technique}
Z.~Palmowski and T.~Rolski,
A technique for exponential change of measure for {Markov} processes,
Bernoulli 8(6) (2002) 767--785.

\bibitem{peskir2006optimal}
G.~Peskir, A.~N. Shiryaev,
Optimal stopping and free-boundary problems,
Springer, (2006).

\bibitem{peskir2002solving}
G.~Peskir, A.~N. Shiryaev, Solving the {Poisson} disorder problem, in: Advances
  in Finance and Stochastics, Springer, 2002, 295--312.

\bibitem{polunchenko2012state}
A.~S. Polunchenko, A.~G. Tartakovsky, State-of-the-art in sequential
  change-point detection, Methodology and Computing in Applied Probability
  14(3) (2012) 649--684.

\bibitem{pollak2009optimality}
M.~Pollak, A.~G. Tartakovsky, Optimality properties of the
  {Shiryaev}--{Roberts} procedure, Statistica Sinica (2009) 1729--1739.

\bibitem{poor2009quickest}
H.~V. Poor, O.~Hadjiliadis, Quickest detection, Vol.~40, Cambridge University
  Press Cambridge, (2009).

\bibitem{roberts1966comparison}
S. Roberts, A comparison of some control chart procedures, Technometrics 8(3)
  (1966) 411--430.

\bibitem{Sato}
K. Sato, L\'evy processes and infinitely divisible distributions, Cambridge Studies in Advanced Mathematics,
vol. 68, Cambridge University Press, Cambridge, (1999).



\bibitem{shiryaev1961problem}
A.~N. Shiryaev, The problem of the most rapid detection of a disturbance in a
  stationary process, in: Soviet Math. Dokl. 2 (1961) 795--799.

\bibitem{shiryaev1963optimum}
A.~N. Shiryaev, On optimum methods in quickest detection problems, Theory of
  Probability \& Its Applications 8(1) (1963), 22--46.

\bibitem{shiryaev1996probability}
A.~N. Shiryaev,
Probability, volume 95 of graduate texts in mathematics, (1996).

\bibitem{0036-0279-51-4-L15}
A.~N. Shiryaev, Minimax optimality of the method of cumulative sums (cusum) in
  the case of continuous time, Russian Mathematical Surveys 51(4) (1996) 750.

\bibitem{shiryaev2002quickest}
A.~N. Shiryaev, Quickest detection problems in the technical analysis of the
  financial data, in: Mathematical Finance -- Bachelier Congress 2000,
  Springer, (2002), 487--521.


\bibitem{shiryaev2004remark}
A.~N. Shiryaev, A remark on the quickest detection problems, Statistics \&
  Decisions/International mathematical Journal for stochastic methods and
  models 22~(1/2004) (2004) 79--82.

\bibitem{shiryaev2006disorder}
A.~N. Shiryaev, From “disorder” to nonlinear filtering and martingale
  theory, in: Mathematical Events of the Twentieth Century, Springer, (2006),
  371--397.

  \bibitem{shiryaev2007optimal}
A.~N. Shiryaev, Optimal stopping rules, Vol. 8, Springer Science \& Business
  Media, (2007).

\bibitem{shiryaev2010quickest}
A.~N. Shiryaev, Quickest detection problems: {Fifty} years later, Sequential
  Analysis 29(4) (2010) 345--385.

\bibitem{mortalityDatabase}
V.~Shkolnikov, M.~Barbieri, J.~Wilmoth, \href{http://www.mortality.org/}{The
  {Human} {Mortality} {Database}}. Available online:\\
\url{http://www.mortality.org/}

\bibitem{gus}
{Statistics Poland}.
Life expectancy in {Poland}.
Available online:\\
  \url{stat.gov.pl/en/topics/population/life-expectancy/life-expectancy-in-poland,1,3.html}.

\bibitem{ss}
B. Strulovici and M. Szydlowski,
On the smoothness of value functions and the existence of optimal strategies in diffusion models,
Journal of Economic Theory 159 (2015) 1016--1055.

\bibitem{zhitlukhin2013bayesian}
M.~Zhitlukhin and A.~N. Shiryaev,
Bayesian disorder problems on filtered probability spaces,
Theory of Probability \& Its Applications 57(3) (2013) 497--511.



\end{thebibliography}
\end{document}